\documentclass[11pt]{amsart}
\usepackage{color}
\usepackage{amsmath}
\usepackage{amssymb}
\usepackage{amsfonts}
\usepackage{mathrsfs}  
\theoremstyle{plain}
\newtheorem{theorem}{Theorem}[section]

\newtheorem{corollary}[theorem]{Corollary}

\newtheorem{lemma}[theorem]{Lemma}

\newtheorem{proposition}[theorem]{Proposition}
\newtheorem{remark}[theorem]{Remark}
\numberwithin{equation}{section}
\newtheorem*{hyp*}{Hypothesis} 
 
\def\d{\displaystyle}

\begin{document}
\title[Structural formulas for matrix-valued orthogonal polynomials]{Structural formulas for matrix-valued orthogonal polynomials related to $2\times 2$  hypergeometric operators}
\author[C. Calder\'on]{C. Calder\'on}
\author[M. M. Castro]{M. M. Castro}
\email[C. Calder\'on]{celeste.calderon@fce.uncu.edu.ar}
\email[M.M. Castro] {mirta@us.es }
\thanks{The research of the first author was partially supported by CONICET and SECTYP-UNCUYO Argentina grant 06\textbackslash M125. The research of the second author was partially supported by PGC2018-096504-B-C31 (FEDER(EU)/Ministerio de Ciencia e Innovaci\'on-Agencia Estatal de Investigaci\'on), FQM-262 and Feder-US-1254600 (FEDET(EU)/Junta de Anadaluc\'ia). }
%	e Innovaci\'on-Agencia Estatal de Investigaci\'on),​ }
	%FQM-262 and Feder-US-1254600 (FEDET(EU)/Jun\-ta de Anda\-lu\-c\'ia).}
\address[C. Calder\'on]{Facultad de Ciencias Exactas y Naturales,
Universidad Nacional de Cuyo, 5500 Mendoza, Argentina}
\address[M. M. Castro]{Departamento de Matem\'atica Aplicada II and IMUS, Universidad de Sevilla, Escuela Polit\'ecnica Superior, calle Virgen de \'Africa 7, 41011, Sevilla, Spain}
\subjclass[2010]{ 42C05; 47S10; 33C45}
\keywords{Matrix-valued orthogonal polynomials; Matrix-valued differential operators; Rodrigues formula}

\begin{abstract}
	We give some structural formulas for the family of matrix-valued orthogonal polyno\-mials of size $2\times 2$ introduced  by C. Calder\'on et al. in an earlier work, which are common eigenfunctions of a differential operator of hypergeometric type. Specifically, we give a Rodrigues formula that allows us to write this family of polynomials explicitly in terms of the classical Jacobi polynomials, and write, for the sequence of orthonormal polynomials, the three-term recurrence relation and the Christoffel-Darboux identity.
	We obtain  a Pearson equation, which enables us to prove that the sequence of  derivatives  of the orthogonal polynomials is also orthogonal, and to compute a  Rodrigues formula for these polynomials  as well as a matrix-valued differential
	operator having these polynomials as eigenfunctions. 
	We also describe the second-order differential operators of the algebra associated with the weight matrix.
\end{abstract}

% in \cite{CGPSZ19}

\maketitle

\section{Introduction}

In the last few years,  the search for examples of matrix-valued orthogonal polynomials that are common eigenfunctions of a second order differential operator, that is to say, satisfying a {\it bispectral property} in the sense of \cite{DG86}, has received a lot of attention after the seminal work of  A. Dur\'an in \cite{D97}. 

The theory of matrix-valued orthogonal polynomials was started by Krein in 1949 \cite{K49}, \cite{K71} (see also \cite{Atk} and \cite{Be}), in connection with spectral analysis and moment problems. Nevertheless, the first examples of orthogonal matrix polynomials satisfying this extra property  and non reducible to scalar case, appeared more recently in \cite{GPT01,GPT02a, GPT03, G03} and \cite{DG04}. The collection of examples has been growing lately  (see for instance \cite{D2,D3,DG5,DdI1,GdI,PT06,PR08,BCD,KPR12,KPR13, PZ16, KRR17,CGPSZ19}). Moreover, the problem  of giving a general classification of these families of matrix-valued orthogonal polynomials as  solutions of the so called {\it Matrix Bochner Problem} has been also recently addressed in   \cite{CY18} and in \cite{C19} for the special case of  $2\times 2$ hypergeometric matrix differential operators. 

As the case of classical orthogonal polynomials, the families of matrix-valued orthogonal polynomials satisfy many formal properties
such as  structural formulas (see for instance \cite{DG3, DL,D10,BCD,KRR17}), which have been very useful to compute explicitly the orthogonal polynomials related with several of these families. Having these explicit formulas is essential when one is looking for applications of these matrix-valued bispectral polynomials, such as in  the problem of {\it time and band limiting}  over a non-commutative ring and matrix-valued commuting operators,
see \cite{GPZ15,CG15,GPZ17,CG17,CGPZ17,GPZ18}.

Recently, in  \cite{CGPSZ19},  a new family of matrix-valued orthogonal polynomials of size $2\times 2$  was introduced, which are common eigenfunctions of a differential operator of {\it hypergeometric type} (in the sense defined by Juan A. Tirao in \cite{T03}):

\begin{equation*} 
D=
%D(C,U,V)\doteq 
\frac{d^{2}}{dt^{2}}t(1-t)+ \frac{d}{dt}\left( C-tU\right)-V, \quad \text{ with } U,V,C\in
\mathbb{C}
^{2\times 2}.
\end{equation*}%

In particular, the polynomials $(P^{\left( \alpha ,\beta ,v\right) }_n)_{n\geq 0}$ introduced in 
\cite{CGPSZ19}, orthogonal with respect to the weight matrix  $W^{(\alpha,\beta,v)}$ given in (\ref{W}) and (\ref{Wtilde}), are common eigenfunctions of an hypergeometric operator with matrix eigenvalues $\Lambda_n$, which are diagonal matrices with no repetition in their entries. This fact could be especially useful if one intends to use this family of polynomials in the context of {\it time and band limiting}, where the commutativity of the matrix-valued eigenvalues $(\Lambda_n)_n$ could play an important role. 

%In this kind of problem, one looks for commuting operators with simple spectrum, although in the matrix framework there is still much to explore.

In this paper we give some structural formulas for the family of matrix-valued orthogonal polynomials introduced in \cite{CGPSZ19}. In particular, in Section \ref{secRodr} we give a Rodrigues formula (see Theorem \ref{Rodrigues}), which allows us to write this family of polynomials explicitly in terms of the classical Jacobi polynomials (see Corollary \ref{RodriguesJacobi}). 

In Section \ref{orto}, this Rodrigues formula allows us to compute the norms of the sequence of monic orthogonal polynomials and therefore, we can find the coefficients of the three-term recurrence relation and the Christoffel-Darboux identity for the sequence of orthonormal polynomials.

In Section \ref{las_derivadas}, we obtain  a Pearson equation (see Proposition \ref{Pearson}), which allows us to prove that the  sequence of  derivatives of $k$-th order, $k\geq 1$, of the orthogonal polynomials is also orthogonal with respect to the weight matrix given explicitly in Proposition \ref{WKDK}.

In Section \ref{shifff}, following the ideas in \cite[Section 5.1]{KRR17}, we use  the Pearson equation  to give explicit lowering and
rising operators for the sequence of derivatives. Thus, we deduce a Rodrigues formula for these polynomials  and find a  matrix-valued differential
operator that has  these matrix-valued polynomials as common eigenfunctions. 

Finally, in Section \ref{algebra}, we describe the algebra of second order  differential operators associated with the weight matrix $W^{(\alpha,\beta,v)}$ given in (\ref{W}) and (\ref{Wtilde}). Indeed, for a given weight matrix $W$, the analysis of the algebra $D(W)$ of all differential operators that have a sequence of matrix-valued orthogonal polynomials with respect to $W$ as eigenfunctions has received much attention in the literature in the last fifteen years \cite{CG06,GT07,T11,PZ16,Z16,C18,CY18}. While for classical orthogonal polynomials the structure of this algebra is very well known (see \cite{LM}), in the matrix setting, where this algebra is non commutative, the situation is highly non trivial.

%In the next Section we include some basic facts concerning weight matrices and
%orthogonal matrix polynomials. 
\section{Preliminaries}

In this section we give some background on matrix-valued orthogonal
polynomials (see \cite{DL1} for further details). A weight matrix $W$ is a complex $N\times N$ matrix-valued
integrable function on the interval $(a,b)$, such that $W$ is positive
definite almost everywhere and with finite moments of all orders, i.e., $ \int_a^b t^{n}dW(t)\in \mathbb{C}^{N \times N}, \ n \in \mathbb{N}$. The weight
matrix $W$ induces a Hermitian sesquilinear form,%
\begin{equation*}
\left\langle P,Q\right\rangle _{W}=\int_{a}^{b}P(t)W\left( t\right)
Q^{\ast }\left( t\right) dt,
\end{equation*}%
for any pair of $N\times N$ matrix-valued functions $P(t)$ and $Q(t)$, where 
$Q^{\ast }(t)$ denotes the conjugate transpose of $Q(t)$.

A sequence $(P_n)_{n\geq 0}$ of orthogonal  polynomials with respect to a weight matrix $W$ is a sequence of matrix-valued polynomials satisfying that $P_n(t)$, $n\ge 0$, is a matrix polynomial of degree $n$ with non-singular leading coefficient, and 
%$\int P_ndWP_m^*=\Delta_n\delta _{n,m}$, 
$\left\langle P_{n},P_{m}\right\rangle _{W}=\Delta_n\delta _{n,m}$, where
$\Delta _n$, $n\ge 0$, is a positive definite matrix.
When $\Delta_n=I$, here  $I$ denotes the identity matrix, we say that the polynomials $(P_n)_{n\geq 0}$ are orthonormal. In particular, when the leading coefficient of $P_n(t)$, $n\ge 0$, is the identity matrix, we say that the polynomials $(P_n)_{n \geq 0}$ are monic. 
%$P_n(t)=It^n+\ldots$, $n\ge 0$ 
% In particular, when  , we have a sequence of monic matrix-valued orthogonal polynomials.
%
%the leading coefficient is the identity matrix

Given a weight matrix $W$, there exists a unique sequence of monic orthogonal
polynomials $\left( P_{n}\right) _{n\geq 0}$ in $%
%TCIMACRO{\U{2102} }%
%BeginExpansion
\mathbb{C}
%EndExpansion
^{N\times N}[t]$, any other sequence of orthogonal polynomials $\left(
Q_{n}\right) _{n\geq 0}$ can be written as 
$Q_{n}(t)=K_{n}P_{n}(t)$ for some
non-singular matrix $K_{n}.$
%=It^n+\ldots$.

Any sequence of monic orthogonal matrix-valued polynomials $\left(
P_{n}\right) _{n\geq 0}$ satisfies a three-term recurrence relation 
\begin{equation*}
tP_{n}(t)=P_{n+1}(t)+B_{n}P_{n}(t)+A_{n}P_{n-1}(t),\quad \text{ for }n\in \mathbb{N}_{0},
\end{equation*}%
where $P_{-1}(t)=0$, $P_{0}(t)=I$. The $N \times N $ matrix  coefficients $A_{n}$ and $B_{n}$ enjoy certain
properties: in particular, $A_{n}$ is non-singular for any $n$. 

Two weights $W$ and $\widetilde{W}$ are said to be \emph{equivalent} if there
exists a non-singular matrix $M$, which does not depend on $t$, such that 
\begin{equation*}
\widetilde{W}(t)=MW(t)M^{\ast },\quad \text{ for all }t\in (a,b).
\end{equation*}%
%{\color{red} quizás no sea necesario incluir esto aquí}
A weight matrix $W$ \textit{reduces} to a smaller size if there exists a
non-singular matrix $M$ such that 
\begin{equation*}%
MW(t)M^{\ast }=%
\begin{pmatrix}
W_{1}(t) & 0 \\ 
0 & W_{2}(t)%
\end{pmatrix}%
,\quad \text{ for all }t\in (a,b),
\end{equation*}%
where $W_{1}$ and $W_{2}$ are weights of smaller size. A weight matrix $W$
is said to be \textit{irreducible} if it does not reduce to a smaller size (see \cite{DG04}, \cite{TZ16}).

Let $D$ be a right-hand side ordinary differential operator with matrix
valued polynomial coefficients, 
\begin{equation*}
D=\sum_{i=0}^{s}\partial ^{i}F_{i}\left( t\right),\text{ \ \ \ \ \ \ \ }%
\partial ^{i}=\frac{d^{i}}{dt^{i}}.
\end{equation*}%
The operator $D$ acts on a polynomial function $P\left( t\right) $ as $%
PD=\sum_{i=0}^{s}\partial ^{i}PF_{i}\left( t\right).$

We say that the differential operator $D$ is symmetric with respect to $W$ if
\begin{equation}\label{symcond}
\left\langle PD,Q\right\rangle _{W}=\left\langle P,QD\right\rangle _{W},\
\textrm{for all}\ P,Q\in 
%TCIMACRO{\U{2102} }%
%BeginExpansion
\mathbb{C}
%EndExpansion
^{N\times N}[t].
\end{equation}

The differential operator $D=\displaystyle \frac{d^{2}}{dt^{2}}F_{2}(t)+\frac{d}{dt}%
F_{1}(t)+F_{0}$ is symmetric with respect to $W$ if and only if ( \cite[Theorem 3.1]{DG04})
\begin{eqnarray}
F_{2}W &=&WF_{2}^{\ast } , \label{EDS1} \\
2\left( F_{2}W\right) ^{\prime } &=&F_{1}W+WF_{1}^{\ast },  \notag\\
\left( F_{2}W\right) ^{\prime \prime }-\left( F_{1}W\right) ^{\prime }
&=&WF_{0}^{\ast }-F_{0}W , \notag
\end{eqnarray}%
and%
\begin{equation}
\lim_{t\rightarrow a,b} F_{2}\left( t\right)W\left( t\right) =0\text{ \ \ \
and \ \ \ }\lim_{t\rightarrow a,b}\left(  F_{1}\left(
t\right)W\left( t\right) -W\left( t\right)F_{1}^{\ast }\left( t\right) \right) =0.
\label{CBorde}
\end{equation}

We will need the following result to find the Rodrigues' formula for the sequence of 
orthogonal polynomials with respect to a weight matrix $W$.

\begin{theorem}
\label{D10R} (\cite[Lemma 1.1]{D10}) Let $F_{2}$, $F_{1}$
and $F_{0}$ be matrix polynomials of degrees not larger than $2$, $1$%
, and $0$, respectively. Let $W$, $R_{n}$ be $N\times N$ matrix functions
twice and $n$ times differentiable, respectively, in an open set of the real
line $\Omega $. Assume that $W(t)$ is non-singular for $t\in $ $\Omega $ and
that satisfies the identity and the differential equations  in (\ref{EDS1}). Define the functions $P_{n}$, $n\geq 1$, by%
\begin{equation*}
P_{n}=R_{n}^{(n)}W^{-1}.
\end{equation*}%
If for a matrix $\Lambda _{n}$, the function $R_{n}$ satisfies%
\begin{equation*}
\left( R_{n}F_{2}^{\ast }\right) ^{\prime \prime }-\left( R_{n}[F_{1}^{\ast
}+n\left( F_{2}^{\ast }\right) ^{\prime }]\right) ^{\prime
}+R_{n}[F_{0}^{\ast }+n\left( F_{1}^{\ast }\right) ^{\prime }+%
\begin{pmatrix}
n \\ 
2%
\end{pmatrix}%
\left( F_{2}^{\ast }\right) ^{^{\prime \prime }}]=\Delta _{n}R_{n},
\end{equation*}%
then the function $P_{n}$ satisfies%
\begin{equation*}
 P_{n}^{^{\prime \prime }}\left( t\right)F_{2}\left( t\right) + P_{n}^{^{\prime }}\left( t\right)F_{1}\left(
 t\right) + P_{n}\left(
t\right)F_{0}\left(
t\right) = \Delta _{n}P_{n}\left(
t\right).
\end{equation*}
\end{theorem}

%the identity (\ref{EDS1}) and the differential equations (\ref{EDS2}) and (\ref{EDS3}).
\subsection{The family of matrix-valued orthogonal polynomials}

In \cite{CGPSZ19} the authors introduce a Jacobi type weight matrix $%
W^{\left( \alpha ,\beta ,v\right) }\left( t\right) $ and a differential
operator $D^{\left( \alpha ,\beta ,v\right) }$ such that $D^{\left( \alpha
,\beta ,v\right) }$ is symmetric with respect to the weight matrix $%
W^{\left( \alpha ,\beta ,v\right) }\left( t\right).\ $

Let $\alpha$, $\beta$, $v\in\mathbb{R}$, $\alpha ,\beta >-1$ and $|\alpha -\beta |<|v|<\alpha +\beta +2$. We consider the weight matrix function 
\begin{equation}
W^{\left( \alpha ,\beta ,v\right) }(t)=t^{\alpha }\left( 1-t\right) ^{\beta
}\,\widetilde{W}^{\left( \alpha ,\beta ,v\right) }\left( t\right), \text{ \ \ for }t\in (0,1),  \label{W}
\end{equation}%
with 
\begin{equation}\label{Wtilde}
\small{\widetilde{W}^{\left( \alpha ,\beta ,v\right) }\left( t\right)} =%
\begin{pmatrix}
\dfrac{v(\kappa_{v,\beta}+2)}{\kappa_{v,-\beta}}t^{2}-\left( \kappa_{v,\beta}+2
\right) t+(\alpha +1) & (\alpha +\beta +2)t-(\alpha +1) \\ 
(\alpha +\beta +2)t-(\alpha +1) & -\dfrac{v(\kappa_{-v,\beta}+2)}{\kappa_{-v,-\beta} }t^{2}-\left( \kappa_{-v,\beta}+2 \right) t+(\alpha +1)%
\end{pmatrix},%
\end{equation}%
where, for the sake of clearness in the rest of the paper, we use the notation:
\begin{equation}
\kappa_{\pm v,\pm\beta}=\alpha \pm v \pm \beta \  .
\end{equation}
%\textrm{and}\  \kappa_{-v,-\beta}=-v+\alpha -\beta.
$W^{\left( \alpha ,\beta ,v\right) }$ is an irreducible weight matrix and
the  hypergeometric type differential operator given by

\begin{equation}\label{Dalpha_beta}
D^{\left( \alpha ,\beta ,v\right) }= \d \frac{d^{2}}{dt^{2}}F_{2}\left( t\right)%
+\frac{d}{dt}F_{1}\left( t\right)+ F_{0}\left( t\right),\end{equation}

 where

\begin{equation}\label{def_F(t)}
F_{2}\left( t\right) =t(1-t),\  F_{1}\left( t\right) =C^{\ast}-tU\  \textrm{and}\ \ F_{0}\left(
t\right) =-V,
\end{equation}

and
\begin{equation}\label{Coef_CUV}
C=%
\begin{pmatrix}
\alpha +1-\dfrac{\kappa_{-v,-\beta}}{v} & \dfrac{\kappa_{v,-\beta}}{v} \\ 
-\dfrac{\kappa_{-v,-\beta}}{v} & \alpha +1+\dfrac{\kappa_{v,-\beta} }{v}%
\end{pmatrix}%
,\ U=\left( \alpha +\beta +4\right)I \ \text{ and }V=%
\begin{pmatrix}
v & 0 \\ 
0 & 0%
\end{pmatrix},
\end{equation}  
is symmetric with respect to the weight matrix $W^{\left( \alpha ,\beta
,v\right) }$.

In the same paper, the authors also give the corresponding monic orthogonal
polynomials in terms of the hypergeometric function $_{2}H_{1}\left( {C,U,V};t\right)$ defined by J. A. Tirao in \cite{T03}  and their three-term recurrence relation.%

\begin{proposition} \cite[Theorem 4.3]{CGPSZ19}
Let $\left( P_{n}^{\left( \alpha ,\beta ,v\right) }\right) _{n\geq 0}$ be the
sequence of matrix-valued monic orthogonal  polynomials associated with the
weight function $W^{\left( \alpha ,\beta ,v\right) }(t)$. Then, $P_{n}^{\left(
\alpha ,\beta ,v\right) }$ is an eigenfunction of the differential operator $%
D^{\left( \alpha ,\beta ,v\right) }$ with diagonal eigenvalue 
\begin{equation}
\Lambda _{n}=%
\begin{pmatrix}
\lambda _{n} & 0 \\ 
0 & \mu _{n}%
\end{pmatrix}%
,\text{ \ \ \ \ \ \ }%
\begin{array}{c}
\lambda _{n}=-n(n-1)-n\left( \alpha +\beta +4\right) -v, \\ 
\mu _{n}=-n(n-1)-n\left( \alpha +\beta +4\right). %
\end{array}
\label{LN}
\end{equation}%

Moreover (\cite[Section 4.2] {CGPSZ19}), the matrix-valued monic orthogonal polynomials $\left( P_{n}^{\left( \alpha ,\beta ,v\right) }\right) _{n\geq 0}$ are given by 

\begin{align}
 \left(P_{n}^{\left( \alpha ,\beta ,v\right)} \left(t \right)\right)^{\ast} =& \,_{2}H_{1}\left( {C,U,V+\lambda _{n}I}%
;t\right)n! \left[ C,U,V+\lambda _{n}I\right] _{n}^{-1}%
\begin{pmatrix}
1 & 0 \\ 
0 & 0%
\end{pmatrix}
\label{h21po} \\
& \quad +\,_{2}H_{1}\left( {C,U,V+\mu _{n}I};t\right)n!\left[ C,U,V+\mu _{n}I%
\right] _{n}^{-1}%
\begin{pmatrix}
0 & 0 \\ 
0 & 1%
\end{pmatrix}%
,  \notag
\end{align}

where
\begin{equation*}
_{2}H_{1}\left( {C,U,V};t\right) =\sum\limits_{k\geq 0}\left[ C,U,V\right]
_{k}\frac{t^{k}}{k!},
\end{equation*}
and $\left[ C,U,V\right] _{k}$ is defined inductively as $\left[ C,U,V%
\right] _{0}=I$ and \newline
\begin{equation}\label{corhyp}
\left[ C,U,V\right] _{k+1}=\left( C+kI\right) ^{-1}\left( k\left( k-1\right)
I+kU+V\right) \left[ C,U,V\right] _{k}.
\end{equation}
\end{proposition}

\begin{proposition}(\cite[Theorem 3.12] {CGPSZ19})\label{recurre_inicial}
The monic orthogonal polynomials $\left( P_{n}^{\left( \alpha ,\beta ,v\right)
}\right) _{n\geq 0}$ satisfy the three-term recurrence relation 
\begin{equation}
tP_{n}^{\left( \alpha ,\beta ,v\right) }(t)=P_{n+1}^{\left( \alpha ,\beta
,v\right) }(t)+B_{n}^{\left( \alpha
,\beta ,v\right) }P_{n}^{\left( \alpha ,\beta ,v\right) }(t)+A_{n}^{\left(
\alpha ,\beta ,v\right) }P_{n-1}^{\left( \alpha ,\beta ,v\right) } (t) \label{RR3}
\end{equation}%
where 
\begin{equation} \label{An} 
A_{n}^{\left( \alpha ,\beta ,v\right) } =a_n^{\left( \alpha ,\beta ,v\right) }\begin{pmatrix}
(4+2n+\kappa_{-v,\beta})(2n+\kappa_{v,\beta}) & 0 \\ 
0 & (4+2n+ \kappa_{v,\beta})(2n+\kappa_{-v,\beta} )%
\end{pmatrix},
\end{equation} with
\begin{equation}
\small{a_n^{\left( \alpha ,\beta ,v\right) }=
\frac{n(1+n+\alpha )(1+n+\beta
)(2+n+\alpha +\beta )}{(1+2n+\alpha +\beta )(2+2n+\alpha +\beta
)^{2}(3+2n+\alpha +\beta )(2+2n+\kappa_{-v,\beta} )(2+2n+\kappa_{v,\beta} )}, \ n \geq 0},%
\end{equation}
%\notag
%\end{eqnarray}
 and the entries of $B_n=B^{\left( \alpha
 	,\beta ,v\right) }_n$, $n\geq0$, are
 \begin{align}\label{losBn}
\left( B_{n}\right) _{11} &=-n\frac{(\alpha+n)v-\kappa_{-v,-\beta}}{
(\alpha+\beta+2n+2)v}+(n+1)\frac{(\alpha+n+1)v-\kappa_{-v,-\beta}}{
(\alpha+\beta+2n+4)v}, \\ \label{Bn}
\left( B_{n}\right) _{21} &={\frac { \kappa_{v,-\beta}  \left( \kappa_{-v,\beta}+2 \right) }{v \left(\kappa_{-v,\beta} +2
\,n+2 \right)  \left( \kappa_{-v,\beta} +2\,n+4 \right) }},\notag \\ 
\left( B_{n}\right) _{12} &={\frac { -\kappa_{-v,-\beta}  \left(\kappa_{v,\beta}+2 \right) }{v \left(\kappa_{v,\beta}+2
\,n+2 \right)  \left( \kappa_{v,\beta}+2\,n+4 \right) }}
, \notag \\
\left( B_{n}\right) _{22} &=-n\frac{(\alpha+n)v+\kappa_{v,-\beta}}{%
(\alpha+\beta+2n+2)v}+(n+1)\frac{(\alpha+n+1)v+\kappa_{v,-\beta}}{%
(\alpha+\beta+2n+4)v} \notag
.
\end{align}
\end{proposition}

%From the norm of the polynomials $\left\Vert Q\right\Vert ^{2}=\left\langle
%Q,Q\right\rangle _{w},$

 Using the symmetry condition (\ref{symcond}) and the three-term recurrence relation (\ref{RR3}) one can easily see that the coefficients $A_{n}^{\left( \alpha ,\beta ,v\right) }$
and $B_{n}^{\left( \alpha ,\beta ,v\right) }$ satisfy the identities:
\begin{eqnarray}
A_{n}^{\left( \alpha ,\beta ,v\right)}\left\Vert P_{n-1}^{\left( \alpha ,\beta ,v\right) }\left( t\right)
\right\Vert ^{2}  &=&\left\Vert
P_{n}^{\left( \alpha ,\beta ,v\right) }\left( t\right) \right\Vert ^{2},
\label{propA} \\
\left(B_{n}^{\left( \alpha ,\beta ,v\right) }\left\Vert P_{n}^{\left( \alpha ,\beta ,v\right) }\left( t\right)
\right\Vert ^{2}\right) ^{\ast }
&=&B_{n}^{\left( \alpha ,\beta ,v\right) }\left\Vert P_{n}^{\left( \alpha ,\beta ,v\right) }\left( t\right)
\right\Vert ^{2} . \label{propB}
\end{eqnarray}
%{\color{red} ¿estas igualdades viene tambi\'en de \cite{CGPSZ19}?}{\color{blue} NO}
\section{Rodrigues formula}\label{secRodr}

In this section we will provide a Rodrigues formula for the sequence of monic
orthogonal polynomials $\left( P_{n}^{\left( \alpha ,\beta ,v\right)
}\right) _{n\geq 0}$ with respect to the weight matrix $W=W^{\left( \alpha
,\beta ,v\right) }$ in (\ref{W}). Moreover, the Rodrigues formula will allow us
to find an explicit expression for the polynomials in terms of Jacobi
polynomials.

\begin{theorem}
\label{Rodrigues} Consider the weight matrix $W(t)=W^{\left( \alpha ,\beta ,v\right) }(t)$ given by the expression in (\ref{W}) and (\ref{Wtilde}). Consider the matrix-valued functions
 $\left( P_{n}\right) _{n\geq 0}$ and $\left( R_{n}\right) _{n\geq 0}$  defined by 
\begin{eqnarray}
P_{n}\left( t\right) &=&\left(
R_{n}\left( t\right) \right) ^{\left(
n\right) }\left(W\left( t\right)\right)^{-1} ,\text{ \ }  \label{Rn} \\
R_{n}(t)=R_{n}^{\left( \alpha ,\beta ,v\right) }\left( t\right) &=&t^{n+\alpha
}\left( 1-t\right) ^{n+\beta }\left( R_{n,2}^{\left( \alpha ,\beta ,v\right)
}t^{2}+R_{n,1}^{\left( \alpha ,\beta ,v\right) }t+R_{n,0}^{\left( \alpha ,\beta
,v\right) }\right), \label{Rn2}
\end{eqnarray}%
with%
\begin{eqnarray*}
R_{n,2}^{\left( \alpha ,\beta ,v\right) } =R_{n,2}&=&%
\begin{pmatrix}
c_{n} & 0  \\
0 & d_{n}
\end{pmatrix}%
,\\ R_{n,1}^{\left( \alpha ,\beta ,v\right) }&=&R_{n,1}=
\frac{1}{v}\begin{pmatrix}
-c_{n}\kappa_{v,-\beta} & \dfrac{c_{n}(\alpha +2n+2+\beta )\kappa_{v,-\beta}}{\left( \kappa_{v,\beta} +2n+2\right) } \\ 
-\dfrac{d_{n}(\alpha +2n+2+\beta )\kappa_{-v,-\beta}}{\left( \kappa_{-v,\beta}
+2n+2\right) } & d_{n}\kappa_{-v,-\beta}%
\end{pmatrix}%
,\text{ } \\
\text{\ \ }R_{n,0}^{\left( \alpha ,\beta ,v\right) }=R_{n,0} &=&\frac{1+n+\alpha}{v}%
\begin{pmatrix}
c_{n}\dfrac{\kappa_{v,-\beta}}{\left( \kappa_{v,\beta} +2n+2\right) } & -c_{n}%
\dfrac{\kappa_{v,-\beta}}{\left( \kappa_{v,\beta} +2n+2\right) } \\ 
d_{n}\dfrac{\kappa_{-v,-\beta}}{\left( \kappa_{-v,\beta} +2n+2\right) } & -d_{n}%
\dfrac{\kappa_{-v,-\beta}}{\left( \kappa_{-v,\beta} +2n+2\right) }%
\end{pmatrix}%
,
\end{eqnarray*}%

where $(c_n)_n$ and $(d_n)_n$ are arbitrary sequences of complex numbers. Then $P_n(t)$
is a polynomial of degree $n$ with non-singular leading coefficient equal to%
\begin{equation*}
\begin{pmatrix}
\dfrac{\kappa_{v,-\beta} \left( \alpha +\beta +n+3\right) _{n}}{%
\left( -1\right) ^{n}v\left( \kappa_{v,\beta} +2\right) }c_{n} & 0 \\ 
0 & \dfrac{\kappa_{-v,-\beta} \left( \alpha +\beta +n+3\right)
_{n}}{\left( -1\right) ^{n+1}v\left( \kappa_{-v,\beta}+2\right) }d_{n}%
\end{pmatrix}%
,
\end{equation*}%
where $(a)_n=a(a+1)\ldots(a+n-1)$ denotes the usual Pochhammer symbol.
Moreover, if we put
\begin{equation}\label{condicion monico}
c_{n}=\frac{\left( -1\right) ^{n}v\left( \kappa_{v,\beta}+2\right) }{\kappa_{v,-\beta} \left( \alpha +\beta +n+3\right) _{n}},\text{ \ \ }%
d_{n}=\frac{\left( -1\right) ^{n+1}v\left( \kappa_{-v,\beta}+2\right) }{%
\kappa_{-v,-\beta} \left( \alpha +\beta +n+3\right) _{n}},
\end{equation}%
then $\left( P_{n}\right) _{n\geq 0}$ is a sequence of monic orthogonal 
polynomials with respect to $W$ and $P_{n}=P_{n}^{\left( \alpha
,\beta ,v\right) }.$
\end{theorem}

\begin{proof}
Let $W$ be the weight matrix given in (\ref{W}) and  $F_{2,}$ $%
F_{1},$ $F_{0}$ and $\Lambda _{n}$ are  the
polynomials coefficients defined in (\ref{def_F(t)})-(\ref{LN}).
% and (\ref{Coef_CUV}). 

%For the sake of clearness we write $R_{n}(t)=R_{n}^{(\alpha,\beta,v)}\left( t\right) $ and $ R_{n,j}=R_{n,j}^{(\alpha,\beta,v)},\ j=0,1,2$.

 Following straightforward computations, we can prove that the matrix-valued function $R_{n}(t)$
%\begin{equation*}
%R_{n}\left( t\right) =t^{n+\alpha }\left( 1-t\right) ^{n+\beta }\left(
%R_{2}t^{2}+R_{1}t+R_{0}\right)
%\end{equation*}%
satisfies the equation 
\begin{equation*}
\left( R_{n}F_{2}^{\ast }\right) ^{\prime \prime }-\left( R_{n}[F_{1}^{\ast
}+n\left( F_{2}^{\ast }\right) ^{\prime }]\right) ^{\prime
}+R_{n}[F_{0}^{\ast }+n\left( F_{1}^{\ast }\right) ^{\prime }+%
\begin{pmatrix}
n \\ 
2%
\end{pmatrix}%
\left( F_{2}^{\ast }\right) ^{^{\prime \prime }}]=\Lambda _{n}R_{n}.
\end{equation*}%
Theorem \ref{D10R} guarantees that the function $P_{n}\left(t\right)=\left( R_{n}\left(
t\right) \right) ^{n} \left(W\left( t\right)\right)^{-1} $ is an eigenfunction of $%
D^{\left( \alpha ,\beta ,v\right) }$ with eigenvalue $\Lambda _{n}$ given in (\ref{LN}).
%=%
%\begin{pmatrix}
%-n(n-1)-n\left( \alpha +\beta +4\right) -v & 0 \\ 
%0 & -n(n-1)-n\left( \alpha +\beta +4\right)%
%\end{pmatrix}%
%.$$
 Then $ P_{n}^{\left( \alpha ,\beta ,v\right) }\left(t\right)$ and 
$P_{n}\left(t\right)$ satisfy the same differential equation.

We will prove that $P_{n}$ is a polynomial of degree $n$ with non-singular
leading coefficient. We will use the following Rodrigues formula for
the classical Jacobi polynomial $p_{n}^{(\alpha ,\beta )}(t)$ (\cite[Chapter IV]{Sz})%
\begin{equation*}
\frac{d^{n}}{dt^{n}}\left[ t^{n+\alpha }\left( 1-t\right) ^{n+\beta }\right]
=n!t^{\alpha }\left( 1-t\right) ^{\beta }p_{n}^{(\alpha ,\beta )}(1-2t),
\end{equation*}%
where%
\begin{equation}
p_{n}^{(\alpha ,\beta )}(1-2t)=\frac{\Gamma \left( n+\alpha+1\right) }{n!\Gamma
\left( n+\alpha+\beta +1\right) }\sum_{j=0}^{n}%
\begin{pmatrix}
n \\ 
j%
\end{pmatrix}%
\frac{\Gamma \left( n+\alpha+\beta +1+j\right) }{\Gamma \left( j+\alpha+1\right) }%
\left( -1\right) ^{j}t^{j}.  \label{PJE}
\end{equation}
Thus, we obtain%
\begin{equation*}
R_{n}^{(n)}\left( t\right) =n!t^{\alpha }\left( 1-t\right) ^{\beta }\left(
p_{n}^{(\alpha +2,\beta )}(1-2t)R_{n,2}t^{2}+p_{n}^{(\alpha +1,\beta
)}(1-2t)R_{n,1}t+p_{n}^{(\alpha ,\beta )}(1-2t)R_{n,0}\right).
\end{equation*}%
We can rewrite $\left(W\left( t\right)\right)^{-1}$ as 
\begin{equation*}
\left(W\left( t\right)\right)^{-1} =t^{-\alpha -2}\left( 1-t\right) ^{-\beta -2}\left(
J_{2}t^{2}+J_{1}t+J_{0}\right),
\end{equation*}%
with%
\begin{eqnarray*}
J_{2} &=&%
\begin{pmatrix}
\dfrac{\kappa_{v,-\beta}}{v(\kappa_{v,\beta} +2)} & 0 \\ 
0 & -\dfrac{\kappa_{-v,-\beta} }{v(\kappa_{-v,\beta}+2)}%
\end{pmatrix},%
\text{ \ } J_{0} =\dfrac{-\kappa_{v,-\beta} \kappa_{-v,-\beta} (\alpha+1)}{v^{2}(\kappa_{v,\beta}+2)(\kappa_{-v,\beta}
+2)}\begin{pmatrix}
1 & 1 \\ 
1 & 1%
\end{pmatrix},\\
J_{1}&=&\dfrac{\kappa_{v,-\beta} \kappa_{-v,-\beta}}{v^{2}} 
\begin{pmatrix}
\dfrac{1}{(\kappa_{v,\beta} +2)} & \dfrac{(\alpha +\beta
+2)}{(\kappa_{v,\beta} +2)(\kappa_{-v,\beta}
 +2)} \\ 
\dfrac{(\alpha +\beta +2)}{(\kappa_{v,\beta} +2)(\kappa_{-v,\beta}+2)} & \dfrac{1}{(\kappa_{-v,\beta}+2)}%
\end{pmatrix}.
\end{eqnarray*}%
Observe that $R_{n,0}J_{0}=0.$ Thus, $P_{n}\left( t\right) $ becomes%
\begin{equation*}
\begin{array}{c}
P_{n}\left( t\right) =n!t^{-1}\left( 1-t\right) ^{-2}\left[p_{n}^{(\alpha+2,\beta )}(1-2t)R_{n,2}t\left( J_{2}t^{2}+J_{1}t+J_{0}\right)\right. \\ 
\quad \quad \quad \quad \quad +\text{ }p_{n}^{(\alpha +1,\beta
)}(1-2t)R_{n,1}\left( J_{2}t^{2}+J_{1}t+J_{0}\right) +p_{n}^{(\alpha ,\beta)}(1-2t)R_{n,0}\left( J_{2}t+J_{1}\right) \left.\right].%
\end{array}%
\end{equation*}%
Hence, $P_{n}\left(t\right)$ is a polynomial of degree $n$ if and only if $t=0$ and $t=1$
are zeros of the following polynomial 
\begin{eqnarray*}
Q\left( t\right) &=&p_{n}^{(\alpha +2,\beta )}(1-2t)R_{n,2}t\left(
J_{2}t^{2}+J_{1}t+J_{0}\right) +\text{ }p_{n}^{(\alpha +1,\beta
)}(1-2t)R_{n,1}\left( J_{2}t^{2}+J_{1}t+J_{0}\right) \\
&&+p_{n}^{(\alpha ,\beta )}(1-2t)R_{n,0}\left( J_{2}t+J_{1}\right) 
\end{eqnarray*}%
and $t=1$ has multiplicity two, i.e., $Q\left( 0\right) =Q\left( 1\right)
=Q^{\prime }\left( 1\right) =0$.

Taking into account that $p_{n}^{(\alpha ,\beta )}(1)=\dfrac{\Gamma \left( n+\alpha
+1\right) }{n!\Gamma \left( \alpha +1\right) }$ and $p_{n}^{(\alpha ,\beta
)}(-1)=(-1)^{n}\dfrac{\Gamma \left( n+\beta +1\right) }{n!\Gamma \left( \beta
+1\right) }$ we have%
\begin{eqnarray*}
Q\left( 0\right) &=&p_{n}^{(\alpha +1,\beta )}(1)R_{n,1}J_{0}+p_{n}^{(\alpha
,\beta )}(1)R_{n,0}J_{1}=\frac{\Gamma \left( n+\alpha +1\right) }{n!\Gamma
\left( \alpha +1\right) }\left( \frac{n+\alpha +1}{\alpha +1}%
R_{n,1}J_{0}+R_{n,0}J_{1}\right) =\mathbf{0}, \\
Q\left( 1\right) &=&\left( -1\right) ^{n}\frac{\Gamma \left( n+\beta
+1\right) }{n!\Gamma \left( \beta +1\right) }(R_{n,2}+R_{n,1}+R_{n,0})\left(
J_{2}+J_{1}+J_{0}\right) =\mathbf{0}.
\end{eqnarray*}%
Now, by taking derivative of $Q\left( t\right) $ with respect to $t$ and
considering the identity$$\left( p_{n}^{(\alpha ,\beta )}\right) ^{\prime }\left( -1\right) =%
 \d \frac{\left( \beta +\alpha +n+1\right) \left( -1\right) ^{n-1}\Gamma \left(
n+\beta +1\right) }{2\left( n-1\right) !\Gamma \left( \beta +2\right), }$$ we
obtain 
%\begin{eqnarray*}
%Q^{\prime }\left( 1\right) &=&-2\left( p_{n}^{(\alpha +2,\beta %)}\right)
%^{\prime }(-1)R_{n,2}\left( J_{2}+J_{1}+J_{0}\right) %+p_{n}^{(\alpha +2,\beta
%)}(-1)R_{n,2}\left( 3J_{2}+2J_{1}+J_{0}\right) \\
%&&-2\left( p_{n}^{(\alpha +1,\beta )}\right) ^{\prime %}(-1)R_{n,1}\left(
%J_{2}+J_{1}+J_{0}\right) +p_{n}^{(\alpha +1,\beta %)}(-1)R_{n,1}\left(
%2J_{2}+J_{1}\right) \\
%&&-2\left( p_{n}^{(\alpha ,\beta )}\right) ^{\prime %}(-1)R_{n,0}\left(
%J_{2}+J_{1}\right) +p_{n}^{(\alpha ,\beta )}(-1)R_{n,0}J_{2}=0.
%\end{eqnarray*}%
\begin{eqnarray*}
	&Q^{\prime }\left( 1\right) =-2\left(\left( p_{n}^{(\alpha +2,\beta )}\right)
	^{\prime }(-1)R_{n,2}+ \left( p_{n}^{(\alpha +1,\beta )}\right)^{\prime }(-1)R_{n,1}+
	\left( p_{n}^{(\alpha ,\beta )}\right) ^{\prime }(-1)R_{n,0}
	\right)
	\left( J_{2}+J_{1}+J_{0}\right)\\
		&+p_{n}^{(\alpha +2,\beta
		)}(-1)R_{n,2}\left( 3J_{2}+2J_{1}+J_{0}\right) 
 +p_{n}^{(\alpha +1,\beta )}(-1)R_{n,1}\left(
	2J_{2}+J_{1}\right)
+p_{n}^{(\alpha ,\beta )}(-1)R_{n,0}J_{2}=0.
\end{eqnarray*}%

This shows that $Q\left( t\right) $ is divisible by $t\left( t-1\right)
^{2}\,$ therefore, $P_{n}\left(t\right)$ is a polynomial of degree $n$ since $\deg \left(
Q(t)\right) =n+3$.

Observe that the leading coefficient of $P_{n}\left(t\right)$ is determined by the
leading coefficient of\\
 $n!p_{n}^{(\alpha +2,\beta )}(1-2t)R_{n,2}J_{2}t^{4}$.
Considering (\ref{PJE}) we have%
\begin{equation*}
\frac{\left( -1\right) ^{n}\Gamma \left( 2n+\alpha+\beta +3\right) }{\Gamma
\left( n+\alpha+\beta +3\right) }R_{n,2}J_{2}=%
\dfrac{\left( -1\right) ^{n}\left( \alpha
+\beta +n+3\right) _{n}}{v }\begin{pmatrix}
\dfrac{ \kappa_{v,-\beta}}{\kappa_{v,\beta}+2 }c_{n} & 0 \\ 
0 & -\dfrac{ \kappa_{-v,-\beta}}{ \kappa_{-v,\beta}+2 }d_{n}%
\end{pmatrix}%
.
\end{equation*}%
The previous matrix coefficient is non-singular since $|\alpha -\beta |<|v|<\alpha +\beta
+2.$

Moreover, if (\ref{condicion monico}) holds true then $P_{n}\left( t\right) $ is a monic polynomial and
equal to $ P_{n}^{\left( \alpha ,\beta ,v\right) }\left(t \right)$.
\end{proof}

\begin{corollary}
	Consider the weight matrix $W^{(\alpha ,\beta ,v)}(t)$  given in (\ref{W}) and (\ref{Wtilde}). Then, the   monic  orthogonal polynomials $P_{n}^{\left(
		\alpha ,\beta ,v\right) }(t)$ satisfy the Rodrigues formula
	
	$${P_{n}^{(\alpha ,\beta ,v)}}(t)= (R_{n}^{(\alpha ,\beta ,v)}(t))^{(n)}\left( W^{(\alpha ,\beta ,v)}(t)\right)^{-1}.$$
\end{corollary}

We can see in the proof of Theorem \ref{Rodrigues} that Rodrigues' formula
allows us to find an explicit expression for the polynomials in terms of the classical
Jacobi polynomials.

\begin{corollary}
\label{RodriguesJacobi} Consider the matrix-valued function $\widetilde{W}^{(\alpha ,\beta ,v)}(t)$  given in (\ref{Wtilde}) and let  $R_{n,i}^{(\alpha ,\beta ,v)}$, $i=0,1,2$, be as in Theorem \ref%
{Rodrigues}. Define the coefficients $c_n$ and $d_n$ as in (\ref{condicion monico}). Then, the sequence of monic orthogonal polynomials $\left(
P_{n}\right) _{n\geq 0},$ defined by (\ref{Rn}) can be written as%
\begin{equation}
\small{P_{n}\left( t\right) =n!\left( p_{n}^{(\alpha
+2,\beta )}(1-2t)R_{n,2}^{\left( \alpha ,\beta ,v\right) }t^{2}+p_{n}^{(\alpha +1,\beta
)}(1-2t)R_{n,1}^{\left( \alpha ,\beta ,v\right) }t+p_{n}^{(\alpha ,\beta )}(1-2t)R_{n,0}^{\left( \alpha ,\beta ,v\right) }\right) (\widetilde{W}^{(\alpha ,\beta ,v)}%
\left( t\right))^{-1} }. \label{PfunJ}
\end{equation}%
Moreover,

\begin{equation}
P_{n}\left( t\right) =n!\left( p_{n}^{(\alpha
,\beta )}(1-2t)\mathscr{C}_{n,2}^{(\alpha,\beta,v)}+p_{n+1}^{(\alpha ,\beta )}(1-2t)\mathscr{C}_{n,1}^{(\alpha,\beta,v)}+p_{n+2}^{(\alpha
,\beta )}(1-2t)\mathscr{C}_{n,0}^{(\alpha,\beta,v)}\right) (\widetilde{W}^{(\alpha ,\beta ,v)}\left( t\right))^{-1}  , \label{PfunJ2}
\end{equation}%
with 
\begin{eqnarray*}
&&\mathscr{C}_{n,2}^{(\alpha,\beta,v)}=\dfrac{\left( \beta +n+1\right) \left( \alpha +n+1\right) }{\left(
\alpha +\beta +2n+2\right) \left( \alpha +\beta +2n+3\right) }%
\begin{pmatrix}
\dfrac{c_{n}\left( \kappa_{-v,\beta} +2n+4\right) }{\kappa_{v,\beta}+2n+2} & 0
\\ 
0 & \dfrac{d_{n}\left( \kappa_{v,\beta} +2n+4\right) }{\kappa_{-v,\beta} +2n+2}%
\end{pmatrix},%
\  \\
&&\mathscr{C}_{n,1}^{(\alpha,\beta,v)}=\frac{n+1}{v}%
\begin{pmatrix}
\dfrac{\left( \alpha -\beta \right) \left( \kappa_{-v,\beta} +2n+4\right) c_{n}%
}{\left( \alpha +\beta +2n+2\right) \left( \alpha +\beta +2n+4\right) } & -%
\dfrac{c_{n} \kappa_{v,-\beta} }{\kappa_{v,\beta}+2n+2} \\ 
\dfrac{d_{n}\kappa_{-v,-\beta} }{\kappa_{-v,\beta}+2n+2} & -\dfrac{%
\left( \alpha -\beta \right) \left( \kappa_{v,\beta}+2n+4\right) d_{n}}{%
\left( \alpha +\beta +2n+2\right) \left( \alpha +\beta +2n+4\right) }%
\end{pmatrix},
\\
&&\mathscr{C}_{n,0}^{(\alpha,\beta,v)}=\frac{\left( n+1\right) \left( n+2\right) }{\left( \alpha +\beta
+2n+4\right) \left( \alpha +\beta +2n+3\right) }%
\begin{pmatrix}
c_{n} & 0 \\ 
0 & d_{n}%
\end{pmatrix}.%
\end{eqnarray*}
\end{corollary}

\begin{proof}
The expression in (\ref{PfunJ}) follows from the  proof above and to obtain (\ref%
{PfunJ2}) we use the following property for the classical Jacobi polynomials  $p_{n}^{(\alpha,\beta )}(t)$ (\cite[Section 4.5]{Sz}):
\begin{equation*}
p_{n}^{(\alpha +1,\beta )}(1-2t)=\frac{\left(
n+\alpha +1\right) p_{n}^{(\alpha ,\beta )}(1-2t)-\left( n+1\right)
p_{n+1}^{(\alpha ,\beta )}(1-2t)}{(2n+\alpha +\beta +2)t}
\end{equation*}%
in (\ref{PfunJ}).
\end{proof}

\section{Orthonormal Polynomials}\label{orto}

In this section we give an explicit expression for the norm of the matrix-valued polynomials $%
\left(P_{n}^{\left( \alpha ,\beta ,v\right) }\right)_{n\geq0} $. In addition, for the sequence of orthonormal polynomials we
show the three-term recurrence relation and the Christoffel-Darboux
formula,  introduced for a
general sequence of matrix-valued orthogonal polynomials in \cite{Dur96} (see also \cite{DL1}).

\begin{proposition}
The norm of the  monic orthogonal polynomials $P_{n}^{\left( \alpha ,\beta
,v\right) }\left( t\right) $, $n\geq 0$, is determined by 
\begin{equation}\label{norm}
\left\Vert P_{n}^{\left( \alpha ,\beta ,v\right) }
\right\Vert ^{2}=\small{\frac{n!vB\left( \alpha +n+2,\beta +n+2\right) }{\left(
		\alpha +n+3+\beta \right) _{n}}}%
\begin{pmatrix}
\dfrac{\left( \kappa_{v,\beta}+2\right) \left( \kappa_{-v,\beta} +2n+4\right) }{%
	\kappa_{v,-\beta} \left( \kappa_{v,\beta} +2n+2\right) } & 0 \\ 
0 & -\dfrac{\left( \kappa_{-v,\beta}+2\right) \left( \kappa_{v,\beta}+2n+4\right) }{\kappa_{-v,-\beta} \left( \kappa_{-v,\beta}+2n+2\right) }%
\end{pmatrix}.%
\end{equation}%
Therefore the sequence of polynomials 
\begin{equation*}
\widetilde{P}_{n}^{\left( \alpha ,\beta ,v\right) }\left( t\right)
=\left\Vert P_{n}^{\left( \alpha
	,\beta ,v\right) }\right\Vert ^{-1}P_{n}^{\left( \alpha ,\beta ,v\right) }(t)
\end{equation*}
is orthonormal with respect to $W.$
\end{proposition}

\begin{proof}
Let $P_{n}^{\left( \alpha ,\beta ,v\right) }\left( t\right)
=\sum_{k=0}^{n}  \mathcal{P}_{n}^{k} t^{k},$ using  Rodrigues' formula we have
\begin{equation*}
\Vert P_{n}^{\left( \alpha ,\beta ,v\right) }\Vert
^{2}=\int_{0}^{1} P_{n}^{\left( \alpha ,\beta ,v\right) }\left(
t\right)  W\left( t\right)\left( P_{n}^{\left( \alpha ,\beta
,v\right) }\left( t\right)\right)^{\ast } dt=\sum_{k=0}^{n}\int_{0}^{1}\left( R_{n}\left(
t\right) \right) ^{(n)}\left(\mathcal{P}_{n}^{k}\right)^{\ast}t^{k}dt.
\end{equation*}

Integrating by parts $n$ times we have, 
\begin{eqnarray*}
\left\Vert P_{n}^{\left( \alpha ,\beta ,v\right) }
\right\Vert ^{2} &=&\left( -1\right)
^{n}\sum_{k=0}^{n}\int_{0}^{1}R_{n}\left( t\right) \frac{d^{n}}{dt^{n}}%
\left(\mathcal{P}_{n}^{k}\right)^{\ast}t^{k}dt=\left( -1\right) ^{n}\int_{0}^{1}R_{n}\left( t\right) \frac{%
d^{n}}{dt^{n}}t^{n}dt \\
&=&\left( -1\right) ^{n}n!\int_{0}^{1}R_{n}\left( t\right) dt=\left(
-1\right) ^{n}n!\int_{0}^{1}t^{n+\alpha }\left( 1-t\right) ^{n+\beta }\left(
R_{n,2}t^{2}+R_{n,1}t+R_{n,0}\right) dt \\
&=&\left( -1\right) ^{n}n!\left[ B\left( \alpha +n+3,\beta +n+1\right)
R_{n,2}+B\left( \alpha +n+2,\beta +n+1\right) R_{n,1}\right. \\
&&\left. +B\left( \alpha +n+1,\beta +n+1\right) R_{n,0}\right],
\end{eqnarray*}%
where $B\left( x,y\right) =\int_{0}^{1}t^{x-1}\left( 1-t\right) ^{y-1} dt$ is
the Beta function. Using the following property,%
\begin{equation*}
B\left( x+1,y\right) =\frac{x}{x+y}B\left( x,y\right)
\end{equation*}%
we obtain,%
\begin{eqnarray*}
\left\Vert P_{n}^{\left( \alpha ,\beta ,v\right) }
\right\Vert ^{2} &=&\left( -1\right) ^{n}n!B\left( \alpha +n+1,\beta
+n+1\right) \\
&&\left( \frac{\left( \alpha +n+1\right) \left( \alpha +n+2\right) }{\left(
\alpha +\beta +2n+2\right) \left( \alpha +\beta +2n+3\right) }R_{n,2}+\frac{%
\alpha +n+1}{\alpha +\beta +2n+2}R_{n,1}+R_{n,0}\right) .
\end{eqnarray*}%
Using the expressions in (\ref{Rn2}), after some straightforward  computations we complete the proof.
\end{proof}

The sequence of orthonormal polynomials satisfies the following properties.

\begin{proposition}
The orthonormal polynomials $\left( \widetilde{P}_{n}^{\left( \alpha ,\beta
,v\right) }\right) _{n\geq 0}$ satisfy the three-term recurrence relation%
\begin{equation}
t\widetilde{P}_{n}^{\left( \alpha ,\beta ,v\right) }(t)=\widetilde{A}^{\left( \alpha ,\beta
	,v\right) }_{n+1}\widetilde{P}%
_{n+1}^{\left( \alpha ,\beta ,v\right) }(t)+\widetilde{B}^{\left( \alpha ,\beta
	,v\right) }_{n}\widetilde{P}%
_{n}^{\left( \alpha ,\beta ,v\right) }(t)+\left(\widetilde{A}^{\left( \alpha ,\beta
	,v\right) }_{n}\right)^{\ast }\widetilde{P}%
_{n-1}^{\left( \alpha ,\beta ,v\right) }(t),
\label{ron}
\end{equation}%
with 
\begin{eqnarray*}
\widetilde{A}^{\left( \alpha ,\beta
	,v\right) }_{n+1} &=& \left\Vert P_{n}^{\left( \alpha ,\beta
	,v\right) } \right\Vert ^{-1}\left\Vert P_{n+1}^{\left( \alpha ,\beta ,v\right)
}\right\Vert, \\
\widetilde{B}^{\left( \alpha ,\beta
	,v\right) }_{n} &=&\left\Vert P_{n}^{\left( \alpha ,\beta ,v\right)
} \right\Vert^{-1} B^{\left( \alpha ,\beta
,v\right) }_{n}\left\Vert P_{n}^{\left( \alpha ,\beta
,v\right) } \right\Vert ,
\end{eqnarray*}%
where $B^{\left( \alpha ,\beta
	,v\right) }_{n}$ is the coefficient of the three-term recurrence relation for  the monic orthogonal polynomials $\left( P_{n}^{\left( \alpha ,\beta ,v\right) }\right)_{n\geq 0}$ (\ref{RR3}). Clearly, $\widetilde{B}_{n}^{\left( \alpha ,\beta
	,v\right) }$ is a symmetric matrix.

%Moreover $\widetilde{B}^{\left( \alpha ,\beta
%	,v\right) }%
%_{n}=\left(\widetilde{B}^{\left( \alpha ,\beta
%	,v\right) }_{n}\right)^{\ast }$.
\end{proposition}

\begin{proof}
By replacing the identity $\widetilde{P}_{n}^{\left( \alpha ,\beta ,v\right)
}\left( t\right) =\left\Vert
P_{n}^{\left( \alpha ,\beta ,v\right) } \right\Vert ^{-1}P_{n}^{\left( \alpha ,\beta ,v\right) }(t)$ in
the three-term recurrence relation (\ref{RR3}) and using  identity (%
\ref{propA}) we obtain  (\ref{ron}), and by (\ref{propB}) one verifies that \\
$\left(\widetilde{B}^{\left( \alpha ,\beta ,v\right)
	}_{n}\right)^{\ast }=\widetilde{B}^{\left( \alpha ,\beta,v\right) }_{n}$.
		%,v\right) }_{n}
	%}_{n}\right)^{\ast }=%
%\begin{equation*}
%\left(\widetilde{B}^{\left( \alpha ,\beta ,v\right)
%}_{n}\right)^{\ast }=\left\Vert P_{n}^{\left( \alpha ,\beta ,v\right)
%} \right\Vert\left(B^{\left( \alpha ,\beta
%,v\right) }_{n}\right)^{\ast }\left\Vert P_{n}^{\left(
%\alpha ,\beta ,v\right) } \right\Vert  ^{-1} =\widetilde{B}^{\left( \alpha ,\beta
%,v\right) }_{n}.
%\end{equation*}
\end{proof}

We also have the following Christoffel-Darboux formula for the
sequence of orthonormal polynomials $\left( \widetilde{P}_{n}^{\left( \alpha
,\beta ,v\right) }\right) _{n\geq 0}$:

\begin{eqnarray*}
&&\sum_{k=0}^{n}\left(\widetilde{P}_{k}^{\left( \alpha ,\beta ,v\right) }\right)^{\ast }\left(
y\right) \widetilde{P}_{k}^{\left( \alpha ,\beta ,v\right) }
\left( x\right) = 
\frac{\left(\widetilde{P}_{n}^{\left( \alpha ,\beta ,v\right) }\right)^{\ast }\left( y\right) 
\left(\widetilde{A}_{n+1}^{\left( \alpha ,\beta ,v\right)
}\right)^{\ast } \widetilde{P}_{n+1}^{\left( \alpha ,\beta
,v\right) }\left( x\right)}{x-y}\\
&& \quad \quad \quad \quad \quad \quad \quad \quad -\frac{\left(\widetilde{P}_{n+1}^{\left(
\alpha ,\beta ,v\right) }\right)^{\ast }\left( y\right) \widetilde{A}^{\left( \alpha ,\beta ,v\right)
}_{n+1}
\widetilde{P}_{n}^{\left( \alpha ,\beta ,v\right) }\left(
x\right)}{x-y}.
\end{eqnarray*}%
Hence, the sequence of monic polynomials $\left( P_{n}^{\left( \alpha ,\beta ,v\right) }\right) _{n\geq 0}$ 
satisfies %
\begin{eqnarray*}
&&\sum_{k=0}^{n}\left(P_{k}^{\left( \alpha ,\beta ,v\right) }\right)^{\ast }\left( y\right)
\left\Vert P_{k}^{\left( \alpha ,\beta ,v\right) }
\right\Vert ^{-2}P_{k}^{\left( \alpha ,\beta ,v\right)
}\left( x\right) =\frac{\left(P_{n}^{\left( \alpha ,\beta
,v\right) }\right)^{\ast }\left( y\right) \left\Vert P_{n}^{\left( \alpha ,\beta ,v\right)
} \right\Vert ^{-2} P_{n+1}^{\left( \alpha
,\beta ,v\right) }\left( x\right) }{x-y}\\
&&\quad \quad \quad \quad \quad \quad \quad \quad -\frac{\left(P_{n+1}^{\left(
\alpha ,\beta ,v\right) }\right)^{\ast }\left( y\right) \left\Vert P_{n}^{\left( \alpha
,\beta ,v\right) } \right\Vert ^{-2} P%
_{n}^{\left( \alpha ,\beta ,v\right) }\left( x\right) }{x-y},
\end{eqnarray*}
where the explicit expression of $\left\Vert P_{n}^{\left( \alpha
	,\beta ,v\right) } \right\Vert ^{-2}$ follows from (\ref{norm}).

\section{The Derivatives of the Orthogonal Matrix-Valued Polynomials}\label{las_derivadas}

In this section we prove that polynomials in the sequence of  derivatives of the
orthogonal matrix polynomials $\left( P_{n}^{\left( \alpha ,\beta ,v\right) }\right) _{n \geq 0}$ are
also orthogonal by obtaining a Pearson equation for the weight
matrix $W^{\left( \alpha ,\beta ,v\right) }(t).$

Let  $\d \frac{d^{k}}{dt^{k}}P_{n}^{\left( \alpha ,\beta ,v\right) }\left(
t\right) $ be the derivative of order $k$ of the monic polynomial $P_{n}^{\left( \alpha ,\beta ,v\right)
}(t) $, for $n\geq k$.  Then \begin{equation}\label{Plsderiv}P_{n}^{\left( \alpha ,\beta ,v,k \right)}\left(
t\right)=\d \frac{\left( n-k\right) !}{n!}\frac{d^{k}}{dt^{k}}P_{n}^{\left(
\alpha ,\beta ,v\right) }\left( t\right) \end{equation} are monic polynomials of degree $n-k$ for all $%
n \geq k.$ 
%We defined $\left( P_{n}^{\left( \alpha ,\beta ,v\right) }\left(
%t\right) \right) ^{\left( k\right) }=\frac{\left( n-k\right) !}{n!}\frac{%
%d^{k}}{dt^{k}}P_{n}^{\left( \alpha ,\beta ,v,0\right) }\left( t\right) $ and
%to ease the notation we denote $P_{n}^{\left( \alpha ,\beta ,v,k\right)
%}\left( t\right) =\left( P_{n}^{\left( \alpha ,\beta ,v\right) }\left(
%t\right) \right) ^{\left( k\right) }.$

The polynomial $P_{n}^{\left( \alpha ,\beta ,v\right) }\left( t\right) $ is an
eigenfunction of the operator $D^{\left( \alpha ,\beta ,v\right) }$ given above in (\ref{Dalpha_beta})-(\ref{Coef_CUV}).
%\begin{equation*}
%D^{\left( \alpha ,\beta ,v\right) }=\frac{d^{2}}{dt^{2}}t(1-t)+\frac{d}{dt}\left( 
%\begin{pmatrix}
%\alpha +2-\frac{\alpha -\beta }{v} & \frac{v-\alpha +\beta }{v} \\ 
%\frac{v+\alpha -\beta }{v} & \alpha +2+\frac{\alpha -\beta }{v}%
%\end{pmatrix}%
%-t\text{ }\left( \alpha +\beta +4\right) \mathrm{\ I}\right) -%
%\begin{pmatrix}
%v & 0 \\ 
%0 & 0%
%\end{pmatrix}%
%\end{equation*}% 

Taking derivative $k$ times we have that $P_{n}^{\left( \alpha ,\beta ,v,k\right)
}\left( t\right) $ is an eigenfunction of the differential hypergeometric
operator%
\begin{equation}
D^{\left( k\right) }=D^{\left( \alpha ,\beta ,v,k\right) }=\frac{d^{2}}{dt^{2}}t(1-t)%
+\frac{d}{dt}((C^{(k)})^{\ast}-tU^{(k)})-V  ,\label{Dabvk}
\end{equation}%
with 
\begin{equation*}
C^{(k)} =C+kI
,\qquad U^{(k)}=U+2kI=\left( \alpha +\beta +4+2k\right) \mathrm{\ I},
%\\
%\qquad
%V^{k}=%
%\begin{pmatrix}
%	v & 0 \\ 
%	0 & 0%
%\end{pmatrix}.
% V^{(k)} &=&%
%\begin{pmatrix}
%v+k(\alpha +\beta +3+k) & 0 \\ 
%0 & k(\alpha +\beta +3+k)%
%\end{pmatrix}%
%\text{ for all }k\geq 0,
\end{equation*}%
where $C$, $U$ and $V$ are the  matrix entries of the operator $D^{\left( \alpha ,\beta ,v\right)}$ given in (\ref{Coef_CUV}).
One has that  $$P_{n}^{\left( \alpha ,\beta
,v,k\right) }D^{\left( \alpha ,\beta ,v,k\right) }=\Lambda _{n}^{(k)}P_{n}^{\left( \alpha ,\beta ,v,k\right) },\quad n\geq k,$$
where $\Lambda _{n}^{(k)}=\Lambda _{n}+kU+k(k-1)I$, with $\Lambda _{n}$ given in (\ref{LN}). One has in particular, the standard expression for the eigenvalue shown in \cite[Proposition 3.3]{CGPSZ19}, $\Lambda_{n}^{(k)}=-(n-k)(n-k-1)I-(n-k)U^{(k)}-V$. More explicitly, \begin{equation}\label{autov_k}
\Lambda _{n}^{(k)}=%
\begin{pmatrix}
\lambda _{n}^{(k)} & 0 \\ 
0 & \mu _{n}^{(k)}%
\end{pmatrix}%
,\text{ \ \ \ \ \ \ }%
\begin{array}{c}
\lambda _{n}^{(k)}=-(n-k)\left(\alpha +\beta +3+n+k\right) -v, \\ 
\mu _{n}^{(k)}=-(n-k)\left( \alpha +\beta +3+n+k\right). %
\end{array}
\end{equation}%%{\color{red} al dejar V constante esto ya no es cierto, pues sería ahora $\Lambda _{n}^{(k)}=
	%\Lambda _{n}-kU-k(k-1)$}.
%{\color{red} $\Lambda _{n}^{k}\,$ al final no depende de k?}
%\begin{remark}\label{RemarkDk}

% we will assume that, 
%\begin{eqnarray}
%V^{k}=%
%\begin{pmatrix}
%v & 0 \\ 
%0 & 0%
%\end{pmatrix}%
%.
%\%end{equation*}%
%Taking this into account we point out the equality $D^{\left( \alpha ,\beta
%,v,k\right) }=D^{\left( \alpha +k,\beta +k,v\right) }.$
%\end{remark}

\begin{remark}\label{derik}
	One notices that $D^{\left( \alpha ,\beta
		,v,k\right) }=D^{\left( \alpha +k,\beta +k,v\right) }.$ Thus, the sequence of derivatives are still common eigenfunctions of an hypergeometric operator with diagonal matrix eigenvalues $\Lambda^{(k)}_n$,  with no repetition in their entries. 
	\end{remark}

 %one has the following expression in terms of the hypergeometric function $_{2}H_{1}\left( {U,V,C};t\right)$ 

\begin{proposition} 
	As in (\ref{h21po}), we have the following explicit expression for the sequence of polynomials 
	$%
	\left( P_{n}^{\left( \alpha ,\beta ,v,k\right) }\right) _{n\geq k}$ in terms
	of hypergeometric functions
	\begin{align}
\left(P_{n}^{\left( \alpha ,\beta ,v,k\right) }\left( t\right)\right)^{\ast} &
	=\,_{2}H_{1}\left( { C^{(k)},U^{(k)},V+\lambda _{n}^{(k)}I}%
	;t\right)
	%\,_{2}H_{1}\left( 
	%\begin{array}{c}
	%\begin{array}{cc}
	%U^{(k)}, & V+\lambda _{n}^{(k)}%
	%\end{array}
	%\\ 
	%C^{(k)}%
	%\end{array}%
	%,t\right) 
	\left( n-k\right) !\left[ C^{(k)},U^{(k)},V+\lambda _{n}^{(k)}I\right]
	_{n-k}^{-1}%
	\begin{pmatrix}
	1 & 0 \\ 
	0 & 0%
	\end{pmatrix}%
	+  \label{QHk} \\ 
	& \,_{2}H_{1}\left( { C^{(k)},U^{(k)},V+\mu _{n}^{(k)}I}
	;t\right)
	%\,_{2}H_{1}\left( 
	%\begin{array}{c}
	%\begin{array}{cc}
	%U^{(k)}, & V^{(k)}+\mu _{n}^{(k)}%
	%\end{array}
	%\\ 
	%C^{(k)}%
	%\end{array}%
	%,t\right) 
	\left( n-k\right) !\left[ C^{(k)},U^{(k)},V+\mu _{n}^{(k)}I\right]
	_{n-k}^{-1}%
	\begin{pmatrix}
	0 & 0 \\ 
	0 & 1%
	\end{pmatrix}%
	, \notag
	\end{align}
	where $C^{(k)}$, $U^{(k)}$ and $V$ are the entries of the differential operator in (\ref{Dabvk}) and $\lambda _{n}^{(k)}$ and $\mu_n^{(k)}$ are the diagonal entries of the matrix eigenvalue $\Lambda_n^{(k)}$ given in (\ref{autov_k}).
\end{proposition}

We include the proof for completeness.

\medskip

\begin{proof}
	Indeed,  the polynomials $\left(P_{n}^{\left( \alpha ,\beta ,v,k\right) }\right)_{n\geq k}$ are common eigenfuntions of the  matrix hypergeometric type operator  
	(\ref{Dabvk}) with diagonal eigenvalue $\Lambda_{n}^{(k)}$. 
	
	The fact that the eigenvalue is diagonal implies that the matrix equation can be written as two vectorial hypergeometric equations as in \cite[Theorem 5]{T03} and the solutions of these equations are the columns of $\left(P_{n}^{\left( \alpha ,\beta ,v,k\right) }\right)_{n\geq k}$.  
	Since the eigenvalues of the matrices $C^{(k)}$, $3+\alpha+k$ and $1+\alpha+k$, are non negative integers for all $k\geq 1$, then these solutions are hypergeometric vector functions. 
	
	Moreover, the vectorial functions are polynomials of degree $n-k$ since the form of the factor 
	$\left((n-k)(n-k-1)I+(n-k)U^{(k)}+V+\lambda _{n}^{(k)}I\right)=-\Lambda_n^{(k)}+\lambda _{n}^{(k)}I$ appearing in the expression of $\left[C^{(k)},U^{(k)},V+\lambda _{n}^{(k)}I\right] _{n-k+1}$ (see \ref{corhyp}), makes its first column equal to zero. Analogously for the second column  of 	$\left[C^{(k)},U^{(k)},V+\mu _{n}^{(k)}I\right] _{n-k+1}$.
	%$$\left((n-k)(n-k-1)I+(n-k)U^{(k)}+V+\lambda _{n}^{(k)}I\right)\ \textrm{and} \ \left((n-k)(n-k-1)I+(n-k)U^{(k)}+V+\mu _{n}^{(k)}I\right)$$ appearing in the expressions of $\left[
%	C^{(k)},U^{(k)},V+\mu _{n+k}^{(k)}I\right] _{n-k+1}$ and $\left[ C^{(k)},U^{(k)},V+%
%	\lambda _{n+k}^{(k)}I\right] _{n-k+1}$ res\-pectively (see \ref{corhyp}), are equal to zero. 
	
	The matrices $\left[
	C^{(k)},U^{(k)},V+\mu _{n}^{(k)}I\right] _{n-k}$ and $\left[ C^{(k)},U^{(k)},V+\lambda _{n}^{(k)}I\right] _{n-k}$ are non-singular, since $\lambda	_{q}^{(k)}\neq \mu _{\ell }^{(k)}$, $\lambda _{q}^{(k)}\neq \lambda
	_{\ell }^{(k)}$ and $\mu _{q}^{(k)}\neq \mu _{\ell }^{(k)}$ for all $q\neq \ell $.

\end{proof}

\begin{proposition}
\label{WKDK} Let $\alpha ,\beta >-(k+1)$ and $|\alpha -\beta |<|v|<\alpha
+\beta +2\left( k+1\right) $. We write
\begin{equation}
W^{(k)}(t)=W^{\left( \alpha ,\beta ,v,k\right) }(t)=t^{\alpha +k}\left( 1-t\right) ^{\beta +k} \widetilde{W}^{(\alpha, \beta, v, k)}\left(
t\right), \textrm{where}\ \small{\widetilde{W}^{\left( \alpha ,\beta ,v,k\right) }\left( t\right)
=W_{2}^{(k)}t^{2}+W_{1}^{(k)}t+W_{0}^{(k)}}, \label{wk}
\end{equation}%
with%

\begin{eqnarray*}
W_{2}^{(k)} &=&%
v\begin{pmatrix}
\dfrac{\kappa_{v,\beta}+2\left( k+1\right) }{\kappa_{v,-\beta} } & 0 \\ 
0 & -\dfrac{ \kappa_{-v,\beta}+2( k+1)}{\kappa_{-v,-\beta} }%
\end{pmatrix}%
,\text{ \ }W_{0}^{(k)}=(\alpha +k+1)%
\begin{pmatrix}
1 & -1 \\ 
-1 & 1.%
\end{pmatrix},
\\
W_{1}^{(k)} &=&%
\begin{pmatrix}
- \kappa_{v,\beta}  & \alpha +\beta  \\ 
\alpha +\beta& - \kappa_{-v,\beta}
\end{pmatrix}+2\left( k+1\right)\begin{pmatrix}
1 & 1  \\ 
1& 1
\end{pmatrix}.
\end{eqnarray*}%
Then $W^{(k)}$ is an irreducible weight matrix and the differential
hypergeometric operator $D^{\left( k\right) }$ in (\ref{Dabvk}
) is symmetric with respect to the weight matrix $W^{(k)}$. Moreover, it holds  that $W^{\left( k\right) }(t)=W^{\left( \alpha +k,\beta +k,v\right) }(t)$.
\end{proposition}

\begin{proof}
	
Taking into account Remark \ref{derik} and the fact that $W^{\left( k\right) }(t)=W^{\left( \alpha +k,\beta +k,v\right) }(t)$,
%As pointed out  above in Remark \ref{RemarkDk} $D^{\left( \alpha ,\beta ,v,k\right) }=D^{\left( \alpha
%+k,\beta +k,v\right) }$.
%Since $P_{n}^{\left( \alpha ,\beta ,v,k\right) }$ is an
%eigenfunction of the operator $D^{\left( \alpha ,\beta ,v,k\right)}$
%with diagonal eigenvalue $\Lambda _{n}$ in (\ref{LN}), corresponding to the operator $D^{\left( \alpha ,\beta ,v \right)}$, one notices that $D^{\left( \alpha ,\beta
%	,v,k\right) }=D^{\left( \alpha +k,\beta +k,v\right) }.$ Moreover, it holds that $W^{\left( \alpha ,\beta
%,v,k\right) }(t)=W^{\left( \alpha +k,\beta +k,v\right) }(t)$.{\color{red} Esto deber\'ia ir como un remark aparte} Therefore,
from Proposition 4.1 in \cite{CGPSZ19} one has that $D^{\left( \alpha +k,\beta
+k,v\right) }$ is symmetric with respect to $W^{\left( \alpha +k,\beta
+k,v\right) }$ and $W^{\left( \alpha +k,\beta +k,v\right) }$ is an
irreducible weight matrix if and only if $\alpha +k$ and $\beta +k$ satisfy $%
\alpha +k>-1$, $\beta +k>-1$ and
$|\left( \alpha +k\right) -\left( \beta
+k\right) |<|v|<\left( \alpha +k\right) +\left( \beta +k\right) +2.$
\end{proof}

We will use the following Pearson equation to prove that the sequence of polynomials {\small$%
\left(P_{n}^{\left( \alpha ,\beta ,v,k\right)}\right)_{n\geq k} $} is orthogonal with respect to $%
W^{(k)}$.

\begin{theorem}
\label{Pearson}The weight matrix $W^{(k)}$ satisfies the following Pearson
equation,%
\begin{eqnarray}  \label{PE}
&&\left( W^{(k)}\left( t\right) \Phi ^{(k)}\left( t\right) \right) ^{\prime }= W^{(k)}\left( t\right)\Psi
^{(k)}\left( t\right) ,\text{ }k\in 
%TCIMACRO{\U{2115} }%
%BeginExpansion
\mathbb{N},
\text{ \ \ with \ \ }\\ \notag
&&\Phi ^{(k)}\left( t\right)
=\mathscr{A}_{2}^{k}t^{2}+\mathscr{A}_{1}^{k}t+\mathscr{A}_{0}^{k}\text{ and }\Psi ^{(k)}\left( t\right)
=\mathscr{B}_{1}^{k}t+\mathscr{B}_{0}^{k},\notag
\end{eqnarray}
where%
%{\color{red} aqui he simplificado varias  expresiones}
\begin{align}
\label{A2}\mathscr{A}_{2}^{k} &=
\begin{pmatrix}
-\dfrac{\kappa_{v,\beta} +2(k+2)}{\kappa_{v,\beta} +2(k+1)} & 0 \\ 
0 & -\dfrac{\kappa_{-v,\beta} +2(k+2)}{\kappa_{-v,\beta}+2(k+1)}%
\end{pmatrix},
\\
\label{A1} \text{ }\mathscr{A}_{1}^{k} &= 
\frac{2}{(\kappa_{-v,\beta}+2(k+1))(\kappa_{v,\beta}+2(k+1))}\begin{pmatrix}
0 & \kappa_{v,-\beta}\\ \kappa_{-v,-\beta} & 0
\end{pmatrix}-\mathscr{A}_{2}^{k}, 
\\ \nonumber &
%&+&\begin{pmatrix}
%	\dfrac{\kappa_{v,\beta} +2(k+2)}{\kappa_{v,\beta} +2(k+1)} & 0\\ 0
%	& 
	%\dfrac{\kappa_{-v,\beta} +2(k+2)}{\kappa_{-v,\beta} +2(k+1)}%
%\end{pmatrix}
\\ \label{A0}
\mathscr{A}_{0}^{k} &=\frac{\kappa_{v,-\beta}\kappa_{-v,-\beta}}{v(\kappa_{-v,\beta}+2(k+1))(\kappa_{v,\beta}+2(k+1))}%
\begin{pmatrix}
-1 & 1 \\ 
-1 & 1%
\end{pmatrix},\\
\label{idBk}
\mathscr{B}_{1}^{k}& =(\alpha +\beta +4+2k)\mathscr{A}_{2}^{k} ,
\\
\label{idB0}
\mathscr{B}_{0}^{k} &=
\left( -(\alpha+k+1)I-\dfrac{1}{v}\begin{pmatrix} 
-\kappa_{-v,-\beta}&0\\ 0&\kappa_{v,-\beta}
\end{pmatrix} \right)\mathscr{A}_{2}^{k}\\
\nonumber &+\dfrac{1}{2v}\left(\dfrac{\alpha+\beta+2k+4}{v}
\mathscr{A}_{1}^{k}+\mathscr{B}_{1}^{k}\right)
\begin{pmatrix} 
-\kappa_{-v,\beta}-2(k+1)&0\\ 0&\kappa_{v,\beta}+2(k+1)
\end{pmatrix}.
\end{align}
%\begin{equation*}
%\mathscr{B}_{0}^{k} =
%\left[ -(\alpha+k+1)I-\dfrac{1}{v}\begin{pmatrix} 
%-\kappa_{-v,-\beta}&0\\ 0&\kappa_{v,-\beta}
%\end{pmatrix} \right]\mathscr{A}_{2}^{k}+\dfrac{\alpha+\beta+2k+4}{v}
%\begin{pmatrix} 
%0&\dfrac{\kappa_{v,-\beta}}{\kappa_{-v,\beta} +2(k+1)}\\ \dfrac{-\kappa_{-v,-\beta}}{\kappa_{v,\beta}+2(k+1)}&0
%\end{pmatrix}
%\end{equation*}
\end{theorem}

\begin{proof}
By replacing the expression of $\Phi ^{(k)}\left( t\right) $ and $\Psi
^{(k)}\left( t\right) $ in (\ref{PE}) and taking derivative we obtain 
\begin{equation}
\left( W^{(k)}\left(
t\right) \right) ^{\prime }\left( \mathscr{A}_{2}^{k}t^{2}+\mathscr{A}_{1}^{k}t+\mathscr{A}_{0}^{k}\right) -W^{(k)}\left( t\right)\left( \left( \mathscr{B}_{1}^{k}-2\mathscr{A}_{2}^{k}\right)
t+\mathscr{B}_{0}^{k}-\mathscr{A}_{1}^{k}\right) =0.  \label{pek}
\end{equation}%
The derivative of $W^{(k)}\left( t\right) $ is
\begin{eqnarray*}
&&\left( W^{(k)}\left( t\right) \right) ^{\prime } =-t^{\alpha +k-1}\left(
1-t\right) ^{\beta +k-1} \left( \alpha +\beta +2k+2\right) 
W_{2}^{(k)}t^{3} + \left( \alpha +k\right) W_{0}^{(k)}
\\
&&+\left[ \left( \alpha +k+2\right) W_{2}^{(k)}-\left( \alpha +\beta
+2k+1\right) W_{1}^{(k)}\right] t^{2}
 + \left[\left( \alpha +k+1\right) W_{1}^{(k)}-\left( \alpha
+\beta +2k\right) W_{0}^{(k)}\right] t .
\end{eqnarray*}%
Hence, the left hand side of (\ref{pek}) is a product between a polynomial of
degree five and\\
 $t^{\left( \alpha +k-1\right) }\left( 1-t\right) ^{\left(
\beta +k-1\right) }.$ Therefore, equating to zero the entries of this polynomial, taking into account (\ref{idBk}) and the equality $W_{0}^{(k)} \mathscr{A}_{0}^{k}=0$, it only remains to verify the  identities below, which follow immediately  by straightforward computations. 
%it is zero if and only if the coefficients
%of the polynomial are zero. By  one easily verifies that the following equations hold true:
\begin{eqnarray*}
\left( \alpha +k+4\right) W_{2}^{(k)}\mathscr{A}_{2}^{k}-\left( \alpha +\beta +2k+3\right)\left(W_{2}^{(k)}
\mathscr{A}_{1}^{k}
+W_{1}^{(k)}\mathscr{A}_{2}^{k}\right)&&\\
+ W_{2}^{(k)}\left(\mathscr{B}_{0}^{k}-\mathscr{B}_{1}^{k}\right)+W_{1}^{(k)}\mathscr{B}_{1}^{k}
&=&0,%
\\
\left( \alpha +k+3\right)\left(W_{2}^{(k)} \mathscr{A}_{1}^{k}+ W_{1}^{(k)}\mathscr{A}_{2}^{k}\right)-\left( \alpha +\beta +2k+2\right)
\left(W_{2}^{(k)}\mathscr{A}_{0}^{k}+W_{1}^{(k)}\mathscr{A}_{1}^{k}\right)
-W_{2}^{(k)}\mathscr{B}_{0}^{k}&&\\+W_{1}^{(k)}\left( \mathscr{B}_{0}^{k}-\mathscr{B}_{1}^{k}\right)
+2W_{0}^{(k)} \mathscr{A}_{2}^{k}
&=&0,
\\
\left( \alpha +k+2\right) \left(W_{2}^{(k)}\mathscr{A}_{0}^{k}+W_{1}^{(k)}  \mathscr{A}_{1}^{k}+W_{0}^{(k)} \mathscr{A}_{2}^{k}\right)-\left( \alpha +\beta +2k+1\right)\left( W_{1}^{(k)}
\mathscr{A}_{0}^{k}+W_{0}^{(k)}\mathscr{A}_{1}^{k}\right)&&\\ 
- W_{1}^{(k)}\mathscr{B}_{0}^{k}+
W_{0}^{(k)}\left( \mathscr{B}_{0}^{k}-\mathscr{B}_{1}^{k}\right) &=&0,%
\\
\left( \alpha +k+1\right)\left( W_{1}^{(k)}\mathscr{A}_{0}^{k}+W_{0}^{(k)} \mathscr{A}_{1}^{k}\right)-W_{0}^{(k)}\mathscr{B}_{0}^{k} &=&0.%
\\
\end{eqnarray*}%

\end{proof}
%{\color{red} He simplificado la expresion de  las ecuaciones y las he comprobado, Todo ok! }

\begin{remark}\label{losW^k}
Let us consider the matrix-valued functions $W^{\left( \alpha ,\beta ,v,k\right) }\left( t\right)=W^{(k)}(t)$, $\Phi ^{(k)}(t)$ and $\Psi ^{(k)}(t)$, $k\in \mathbb{N}$, defined in (\ref{wk}) and Theorem \ref{Pearson} respectively. Then, by straightforward computations, one can verify the
following identities:

%\begin{description}
%\item[a] 
\begin{equation}\label{IdW^k1}
W^{\left( \alpha ,\beta ,v,k+1\right) }\left( t\right) = W^{\left( \alpha ,\beta ,v,k\right) }\left( t\right)\Phi
^{(k)}\left( t\right) ,
\end{equation}

%\item[b]
 \begin{equation}\label{IdW^k2}\left( W^{\left( \alpha ,\beta ,v,k+1\right) }\left( t\right)
\right) ^{\prime }= W^{\left( \alpha ,\beta
,v,k\right) }\left( t\right)\Psi ^{(k)}\left( t\right) ,
\end{equation}
 
%\begin{equation}
%W^{\left( \alpha ,\beta ,v,0\right) }\left( t\right) =W^{\left(
%\alpha ,\beta ,v\right) }(t).
%\end{equation}

%\end{description}
\end{remark}

\medskip

Taking into account that $\deg \left(
\Phi ^{(k)}(t)\right) =2$ and $\deg \left( \Psi ^{(k)}(t)\right) =1$, we obtain  from \cite[Corollary 3.10]{CMV07} the following

\begin{corollary}
The sequence of polynomials $\left( P_{n}^{\left( \alpha ,\beta ,v,k\right)
}\right) _{n\geq k}$ is orthogonal with respect to the weight matrix
$W^{(k)} =W^{\left( \alpha ,\beta ,v,k-1\right) }\left(
t\right)\Phi ^{(k-1)}\left( t\right) .$
\end{corollary}

The following results are obtained in a similar way than in Theorem \ref%
{Rodrigues} and Corollary \ref{RodriguesJacobi}.

\begin{proposition}
Let $W^{(k)}(t)$ be defined as in (\ref{wk}). A Rodrigues formula for the sequence of
polynomials $\left( P_{n}^{\left( \alpha ,\beta ,v,k\right) }\right) _{n\geq
k}$ is%
\begin{equation*}
 P_{n}^{\left( \alpha ,\beta ,v,k\right) }(t)=\left(
R_{n}^{\left( \alpha ,\beta ,v,k\right) }\left( t\right) \right) ^{\left(
n-k\right) }\left(W^{(k)}\left( t\right)\right)^{-1} ,\ \textrm{where}\text{ \ \ }R_{n}^{\left( \alpha ,\beta
,v,k\right) }\left( t\right) =R_{n-k}^{\left( \alpha +k,\beta +k,v\right)
}\left( t\right).
\end{equation*}
\end{proposition}

\begin{corollary}
Let the matrix-valued function $W^{(k)}(t),$ and the matrices $R^{\left( \alpha
	,\beta ,v\right) }_{n-k,2},R^{\left( \alpha
,\beta ,v\right) }_{n-k,1}$ and $R^{\left( \alpha
,\beta ,v\right) }_{n-k,0}$ be defined  as in (\ref{wk}) and (\ref{Rn2}). From the
 Rodrigues formula we get the explicit expressions for the sequence of
polynomials $\left( P_{n}^{\left( \alpha ,\beta ,v,k\right) }\right) _{n\geq
k}$ in terms of the classical Jacobi polynomials $p_{n}^{(\alpha ,\beta )}(t)$,%
\begin{multline*}
 P_{n}^{\left( \alpha ,\beta ,v,k\right) }\left(
t\right) =(n-k)!\left( p_{n-k}^{(\alpha
+2+k,\beta +k)}(1-2t)R_{n-k,2}^{\left( \alpha +k,\beta +k,v\right) }t^{2}  +p_{n-k}^{(\alpha +1+k,\beta +k)}(1-2t)R_{n-k,1}^{\left( \alpha
+k,\beta +k,v\right) }t\right.
\\
\left.+p_{n-k}^{(\alpha +k,\beta +k)}(1-2t)R_{n-k,0}^{\left(
\alpha +k,\beta +k,v\right) }\right) \left(\widetilde{W}^{(k)}\right)^{-1}
\end{multline*}%
and%
\begin{multline*}
P_{n}^{\left( \alpha ,\beta ,v,k\right) }\left(
t\right) =(n-k)!\left( p_{n-k}^{(\alpha
+k,\beta +k)}(1-2t)\mathscr{C}_{n-k,2}^{(\alpha+k,\beta+k,v)}+p_{n-k+1}^{(\alpha +k,\beta
+k)}(1-2t)\mathscr{C}_{n-k,1}^{(\alpha+k,\beta+k,v)}\right.\\
\qquad\qquad\qquad \left. +p_{n-k+2}^{(\alpha +k,\beta +k)}(1-2t)\mathscr{C}_{n-k,0}^{(\alpha+k,\beta+k,v)}\right) \left(\widetilde{W%
}^{(k)}\right)^{-1},
\end{multline*}%
with $\mathscr{C}_{n-k,i}^{\left( \alpha +k,\beta +k,v\right) },$ \ \ $i=0,1,2$, given by (\ref{PfunJ2}).
%$\mathscr{C}_{n-k,i}^{(\alpha,\beta,v,k)}=

%{\color{red}Qu\'e son los coeficientes $\mathscr{C}_{i}^{\left( \alpha +k,\beta +k,v\right)}$? }
%{\color{blue} Modifiqué la notación para que se entienda}
\end{corollary}

\begin{proposition}
The orthogonal monic polynomials $\left( P_{n}^{\left( \alpha ,\beta
,v,k\right) }\right) _{n\geq k}$ satisfy the three-term recurrence relation%
\begin{equation*}
tP_{n}^{\left( \alpha ,\beta ,v,k\right) }(t)=P_{n+1}^{\left( \alpha ,\beta
,v,k\right) }(t)+B_{n}^{\left(
\alpha ,\beta ,v,k\right) }P_{n}^{\left( \alpha ,\beta ,v,k\right) }(t)
+A_{n}^{\left( \alpha ,\beta ,v,k\right) }P_{n-1}^{\left( \alpha ,\beta ,v,k\right)
}(t)
\end{equation*}%
with 
\begin{equation*}
B_{n}^{\left( \alpha ,\beta ,v,k\right) }=B_{n-k}^{\left( \alpha +k,\beta
+k,v\right) },\quad A_{n}^{\left( \alpha ,\beta ,v,k\right) }=A_{n-k}^{\left(
\alpha +k,\beta +k,v\right) },\ n\geq k. 
\end{equation*} 
The explicit expressions of $B_{n}^{\left( \alpha ,\beta
	,v\right) }$ and $A_{n}^{\left( \alpha ,\beta
	,v\right) }$ are given in (\ref{An})-(\ref{losBn}).
\end{proposition}

Considering that $W^{\left( k\right) }(t)=W^{\left( \alpha +k,\beta +k,v\right) }(t)$ (see Proposition \ref{WKDK}), the previous recurrence follows directly from (\ref{recurre_inicial}). Notwithstanding, we include the following proof for completeness.

\medskip

\begin{proof}
If we write $P_{n}^{\left( \alpha ,\beta ,v,k\right) } (t)=\sum_{s=0}%
^{n-k}\mathcal{P}_{n-k}^{s}t^{s}$, from (\ref{QHk}) we have the following explicit
expressions,%
\begin{equation}\label{Pn-k-1}
\mathcal{P}_{n-k}^{n-k-1}=\dfrac{\left( n-k\right)}{v} 
\begin{pmatrix}
-\dfrac{(\alpha+n)v-\kappa_{-v,-\beta}}{(\alpha +\beta +2n+2)} & 
\dfrac{\kappa_{-v,-\beta}}{(\kappa_{v,\beta} +2n+2)} \\ 
-\dfrac{\kappa_{v,-\beta}}{(\kappa_{-v,\beta} +2n+2)} & -\dfrac{(\alpha
+n)v+\kappa_{-v,-\beta}}{(\alpha +\beta +2n+2)}%
\end{pmatrix},%
\end{equation}%

\begin{align}\label{Pn-k-2}
\mathcal{P}_{n-k} ^{n-k-2} &=\dfrac{\left( n-k\right) \left( n-k-1\right)(\alpha+n+1)}
{(\alpha +\beta+2n+2)}
%\gamma^{\alpha,\beta,v,k}_{n}
 \left[\dfrac{\alpha+n}{2(\alpha +\beta+2n+1)}\begin{pmatrix}
\dfrac{\kappa_{v,\beta}+2n}{\kappa_{v,\beta}+2n+2}&0  \\ 
0 &
\dfrac{\kappa_{-v,\beta}+2n}{\kappa_{-v,\beta}+2n+2}
\end{pmatrix} \right.\\
\nonumber
&\qquad\qquad\qquad+\dfrac{n+\beta+1}{\alpha +\beta+2n+1}\begin{pmatrix}
\dfrac{1}{\kappa_{v,\beta}+2n+2}&0  \\ 
0 &
\dfrac{1}{\kappa_{-v,\beta}+2n+2}
\end{pmatrix} \\
&\qquad\qquad\qquad\qquad\qquad\qquad\qquad\qquad+\left. \dfrac{1}{v}\begin{pmatrix} -\dfrac{\alpha -\beta}{(\kappa_{v,\beta}+2n+2)}&-\dfrac{\kappa_{-v,-\beta}}{%
	\kappa_{v,\beta}+2n+2}\\ \dfrac{\kappa_{v,-\beta}}{%
	\kappa_{-v,\beta}+2n+2}& \dfrac{\alpha -\beta}{(\kappa_{-v,\beta}+2n+2)}
\end{pmatrix}\right].
\end{align}%
%where\begin{equation}\label{gammas}
%\gamma^{\alpha,\beta,v,k}_{n}	=\dfrac{\left( n-k\right) \left( n-k-1\right)(\alpha+n+1)}
%{(\alpha +\beta+2n+2)v}
%\end{equation}
If we consider the coefficient of order \thinspace $n-k$ and $n-k-1$ in the
three-term recurrence relation we have, %
\begin{eqnarray*}
B_{n}^{\left( \alpha ,\beta ,v,k\right) } &=&\mathcal{P}_{n-k}
^{n-k-1}-\mathcal{P}_{n+1-k} ^{n-k-1},  \\
A_{n}^{\left( \alpha ,\beta ,v,k\right) }&=& \mathcal{ P}_{n-k}
^{n-k-2}-\mathcal{P}_{n+1-k}^{n-k-1}-B_{n}^{\left( \alpha ,\beta ,v,k\right) }\mathcal{ P}_{n-k}
^{n-k-1}, \qquad  n \in \mathbb{N},
\end{eqnarray*}%
respectively. Comparing with the expressions of $B_{n-k}^{\left( \alpha+k ,\beta+k
	,v\right) }$ and $A_{n-k}^{\left( \alpha+k ,\beta+k
	,v\right) }$ given by substituting properly in (\ref{An})-(\ref{losBn}), the proposition follows.
\end{proof}

\section{Shift operators}\label{shifff}

In this section we use  Pearson equation (\ref{PE}) to give explicit lowering and
rising operators for the monic $n$-degree  polynomials $P_{n+k}^{\left( \alpha ,\beta
,v,k\right) }\left(t\right)$, $n\geq 0$, defined in (\ref{Plsderiv}). Moreover, from the existence of the shift operators we deduce
a Rodrigues formula for the  sequence of derivatives $\left(P_{n+k}^{\left( \alpha ,\beta
	,v,k\right) }\right)_{n\geq 0}$, and we find a  matrix-valued differential
operator for which these matrix-valued polynomials are eigenfunctions. In what follows, we will consider the matrix-valued functions $W^{(k)}\left(t\right),\Phi ^{(k)}\left(t\right)$ and $\Psi ^{(k)}\left(t\right)$, $k\in \mathbb{N}$ as defined above in Theorem  \ref{Pearson}.

For any pair of  matrix-valued functions $P$ and $Q$, we denote%
\begin{equation*}
\left\langle P,Q\right\rangle _{k}=\int_{0}^{1}P\left( t\right)
W^{(k)}\left( t\right) Q^{\ast }\left( t\right) dt.
\end{equation*}

\begin{proposition}
\label{Shift}Let $\eta ^{(k)}$  be the first order matrix-valued right differential
operator \begin{equation}\label{shif}\eta ^{(k)}= \dfrac{d}{dt} (\Phi ^{(k)}\left( t\right))^{\ast}
+ (\Psi ^{(k)}\left( t\right))^{\ast}.\end{equation} Then 
$\dfrac{d}{dt}:L^{2}\left( W^{(k)}\right) \rightarrow L^{2}\left(
W^{(k+1)}\right) $ and $\eta ^{k}:L^{2}\left( W^{(k+1)}\right) \rightarrow
L^{2}\left( W^{(k)}\right) $ satisfy 
\begin{equation*}
\left\langle\frac{dP}{dt},Q\right\rangle _{k+1}=-\left\langle P, Q\eta
^{(k)}\right\rangle _{k}.
\end{equation*}
\end{proposition}

\begin{proof}
From $\left\langle\dfrac{dP}{dt},Q\right\rangle _{k+1}=\d\int_{0}^{1}\dfrac{dP(t)}{dt}
W^{(k+1)}(t)Q^{\ast}(t)dt,$ integrating by parts and taking into account equalities (\ref{IdW^k1}) and (\ref{IdW^k2}) in Remark \ref{losW^k} we get,
\begin{eqnarray*}
\left\langle\dfrac{dP}{dt},Q\right\rangle _{k+1}&=&-\int_{0}^{1}P(t)%
\dfrac{d}{dt}\left(W(t)^{(k+1)}\right) Q^{\ast}(t)dt-\int_{0}^{1}P(t)
W^{(k+1)}(t)\left(\dfrac{dQ(t)}{dt}\right)^{\ast}dt\\ 
&=&-\int_{0}^{1}P(t)W^{(k)}(t)\Psi ^{(k)}(t)Q^{\ast}(t)dt-\int_{0}^{1} P(t) W^{(k)}(t)\Phi ^{(k)}(t) \left(\dfrac{dQ(t)}{dt}\right)^{\ast}dt\\
&=&-\int_{0}^{1}P(t)W^{(k)}(t)\left( \Psi ^{(k)}\left(
t\right)Q^{\ast}(t) + \Phi ^{(k)}\left( t\right)\left(\dfrac{dQ(t)}{dt}\right)^{\ast}
\right) Q^{\ast}(t)dt\\
%=-\int_{0}^{1}\left( \left( \Phi ^{(k)}\left( t\right) \right)
%^{\ast }\frac{dP}{dt}+\left( \Psi ^{(k)}\left( t\right) \right) ^{\ast
%}P\right) ^{\ast }W^{(k)}Q^{\ast}$
&=&-\left\langle P,Q\eta ^{(k)}\right\rangle _{k}.
\end{eqnarray*}
\end{proof}

\begin{lemma}\label{lemmaCnk}
The following identity holds true
$$ I\eta
^{(k+n-1)}\cdots \eta ^{(k+1)}\eta ^{(k)}=\mathcal{C}_{n}^{k}P_{n+k}^{\left( \alpha ,\beta ,v,k\right) }, \quad n \geq 1,$$ for a given $k\geq 0$, where 

\begin{equation}\label{Cnk}
\mathcal{C}_{n}^{k} =\left(
-1\right) ^{n} 
\left( \alpha +\beta +3+2k+n\right) _{n}\begin{pmatrix}
\dfrac{\left( \kappa_{v,\beta} +2(k+1+n)\right) }{\left( \kappa_{v,\beta} +2(k+1)\right) } & 0 \\ 
0 & \dfrac{\left( \kappa_{-v,\beta} +2(k+1+n)\right)}{\left( \kappa_{-v,\beta} +2(k+1)\right) }%
\end{pmatrix}%
,\text{ \ \ }n\geq 1.\ 
%\nu \left( n\right) &=&\left( n+\alpha +\beta +4+2k\right) _{n+1}\left(
%-1\right) ^{n}
\end{equation}
\end{lemma}
%\textrm{\color{red}>no seria $n\geq k-1$?}
%{\color{blue} se podría aclarar $k \leq n+1$}
%{\color{red} Revisar esta demostraci\'on con el nuevo operador shifht a la derecha}
\begin{proof}
%Applying the raising operators $\eta^{(k+n-1)}\cdots\eta ^{(k+1)}\eta ^{(k)}$ to $P_{0}=I$ we get a multiple of $P_{n+k}^{\left( \alpha ,\beta
%,v,k\right) }.$ 
 It holds that  $I\eta^{(k+n-1)}\cdots\eta ^{(k+1)}\eta ^{(k)}$ is a
 polynomial of degree $n$. From  the definition of the monic  sequence of derivatives in (\ref{Plsderiv}) one has  $$\dfrac{d}{dt}P_{n+k}^{(\alpha,\beta,v,k)}\left(t\right)=nP_{n+k}^{(\alpha,\beta,v,k+1)}\left(t\right).$$
 Thus, Proposition \ref{Shift} implies that $P_{n+k}^{(\alpha,\beta,v,k+1)}\eta^{(k)}$ is a multiple of $P_{n+k}^{(\alpha,\beta,v,k)}$.

% {\color{red} lo que habria que probar ahora mismo, para un $k$ fijo, es }
 
% we have that $P_{n+k}^{(\alpha,\beta,v,k)}\eta^{(k+n)}$ is a multiple of $P_{n+k+1}^{(\alpha,\beta,v,k)}$ 
 %{\color{red} intentando una induccion y la verdad es que no lo veo claro }

%{\color{red} explicar de donde sale esto. Adem\'as esto no seria un polinomio de grado n+1-k? }
%Proposition \ref{Shift} implies that..... completar.

Therefore, applying the raising operators $\eta^{(k+n-1)}\cdots\eta ^{(k+1)}\eta ^{(k)}$ to $P_{n+k}^{(\alpha,\beta,v,k+n)}=I$ we get a multiple of $P_{n+k}^{\left( \alpha ,\beta
	,v,k\right) }.$   For  the leading coefficient ${C}_{n}^{k}$ of the polynomial  $I\eta^{(k+n-1)}\cdots \eta ^{(k+1)} \eta ^{(k)}$ one obtains the expression

\begin{equation*}
	\mathcal{C}_{n}^{k} =\prod_{i=1}^{n}\left( \left( i-1\right)
	\mathscr{A}_{2}^{k+n-i}+\mathscr{B}_{1}^{k+n-i}\right) .
	\end{equation*}

The diagonal matrix entries $\mathscr{A}_2^k$ and $\mathscr{B}_1^k$ are defined  in (\ref{A2}) and (\ref{idBk}). 
%In particular, $\mathscr{B}_1^k=\left(\mathscr{B}_1^k\right)^{\ast}$ and $\mathscr{A}_2^k=\left(\mathscr{A}_2^k\right)^{\ast}$.
Then, by replacing $\mathscr{B}_{1}^{k}=\left( \alpha +\beta +4+2k\right) \mathscr{A}_{2}^{k}$ in
the identity above we have%
\begin{eqnarray*}
	\mathcal{C}_{n}^{k}	&=&\prod_{i=1}^{n}\left( \left(
	2n+\alpha +\beta +3+2k-i\right) \mathscr{A}_{2}^{k+n-i}\right) \\
	&=&
	(-1)^{n}\prod_{i=1}^{n}\tiny{\left( 2n+\alpha +\beta +3+2k-i\right)\begin{pmatrix}
		\prod_{i=1}^{n}\tiny{\dfrac{ \left( \kappa_{v,\beta}+2(k+n-i+2)\right) }{\kappa_{v,\beta} +2(k+n-i+1)}} & 0 \\ 
		0 & \prod_{i=1}^{n}\tiny{\dfrac{\left(
			\kappa_{-v,\beta} +2(k+n-i+2)\right) }{\kappa_{-v,\beta}+2(k+n-i+1)}}%
	\end{pmatrix}}.%
\end{eqnarray*}%

Hence, the proof follows.

Note that $\mathcal{C}_{n}^{k}$ is non-singular since $|\alpha -\beta |<|v|<\alpha
+\beta +2\left( k+1\right) .$
\end{proof}

From the proposition and the lemma above we obtain another expression for the
Rodrigues formula.

\begin{proposition}
The polynomials $\left(P_{n+k}^{\left( \alpha ,\beta ,v,k\right) }\right)_{n\geq0}$  satisfy the following Rodrigues formula, 
\begin{equation*}
P_{n+k}^{\left( \alpha ,\beta ,v,k\right) }\left( t\right) =\left( \mathcal{C}_{n}^{k}\right) ^{-1}\left( \dfrac{d^{n}}{dt^{n}}W^{(k+n)}\left( t\right)
\right) \left( W^{(k)}\left(
t\right) \right) ^{-1}, \quad n\geq 1,
\end{equation*}
where the matrices $\mathcal{C}_n^k$ are given by the expression in (\ref{Cnk}).
\end{proposition}
\begin{proof}
Let $Q$  be a matrix-valued function and $\eta ^{(k)}$ the raising operator in (\ref{shif}), then $$%
Q\eta ^{(k)}=\dfrac{dQ}{dt} (\Phi ^{(k)})^{\ast}+Q (\Psi
^{(k)})^{\ast}.$$ Using the identities  (\ref{IdW^k1}) and (\ref{IdW^k2})
% $W^{\left( \alpha ,\beta
%,v,k+1\right) }\left( t\right) =\Phi ^{k}\left( t\right) W^{\left( \alpha
%,\beta ,v,k\right) }\left( t\right)$ and $\left( W^{\left( \alpha ,\beta
%,v,k+1\right) }\left( t\right) \right) ^{\prime }=\Psi ^{k}\left( t\right)
%W^{\left( \alpha ,\beta ,v,k\right) }\left( t\right) $ 
we obtain%
\begin{equation*}
Q\eta ^{(k)}=\frac{d}{dt}\left(Q W^{(k+1)}\right)\left( W^{(k)}\right) ^{-1}.
\end{equation*}%
Iterating, it gives 
\begin{equation*}
 Q\eta ^{(k+n-1)}\cdots\eta ^{(k+1)}\eta ^{(k)}=\dfrac{%
d^{n}}{dt^{n}}\left(Q W^{(k+n)}\right)\left( W^{(k)}\right) ^{-1}.
\end{equation*}%
Now, taking $Q\left( t\right) =I\,$\ and using Lemma \ref{lemmaCnk} we have 
\begin{equation*}
P_{n+k}^{\left( \alpha ,\beta ,v,k\right) }(t)=\left( \mathcal{C}_{n}^{k}\right) ^{-1}\dfrac{%
d^{n}}{dt^{n}}\left( W^{(k+n)}(t)\right) \left( W^{(k)}(t)\right) ^{-1}.
\end{equation*}
\end{proof}

\begin{corollary}
Let $W^{(k)}\left( t\right) $ be the weight matrix (\ref{wk}). Then, the differential
operator
\begin{equation}
E^{(k)}= \frac{d}{dt}\circ \eta ^{(k)}=\frac{d^{2}}{dt^{2}}(\Phi ^{(k)}\left( t\right))^{\ast} + \frac{d}{dt}(\Psi ^{(k)}\left( t\right))^{\ast}%
  \label{OE}
\end{equation}%
is symmetric with respect to $W^{(k)}\left( t\right) $ for all $k\in 
%TCIMACRO{\U{2115} }%
%BeginExpansion
\mathbb{N}
%EndExpansion
_{0}.$ Moreover, the polynomials $\left( P_{n+k}^{\left( \alpha
,\beta ,v,k\right) } \right) _{n\geq 0}$ are eigenfunctions                                                                 of
the operator $E^{(k)}$ with eigenvalue%
\begin{equation*}
\Lambda _{n}\left( E^{(k)}\right) =n(n+\alpha+\beta+3+2k)\mathscr{A}^{k}_{2},
\end{equation*}
%=%
%-n(n+\alpha +\beta +3+2k)\begin{pmatrix}
%\dfrac{ \kappa_{v,\beta}+2(k+2)}{\kappa_{v,\beta}+2(k+1)%
%} & 0 \\ 
%0 & \dfrac{\kappa_{-v,\beta} +2(k+2)}{\kappa_{-v,\beta}
%+2(k+1)}
%\end{pmatrix},%

where $\mathscr{A}^{k}_{2}$ is given by (\ref{A2}).
\end{corollary}

\begin{proof}
From Proposition \ref{Shift} and the factorization $E^{(k)}=\dfrac{d}{dt}\circ \eta ^{(k)}$ follows directly that $E^{(k)}$ is symmetric with respect to $%
W^{(k)}.$

The eigenvalue is obtained by looking at the leading coefficients of $\Phi^{(k)}(t)$ and $\Psi^{(k)}(t)$ in (\ref{PE}). Thus, we obtain $\Lambda
_{n}\left( E^{(k)}\right) =n\left( n-1\right) \mathscr{A}^{k}_{2}+n\mathscr{B}^{k}_{1}=n(n+\alpha+\beta+3+2k)\mathscr{A}^{k}_{2}$.
\end{proof}
%{\color{red} leading coefficients of what...at this point...?}
\begin{remark}
The operators $E^{(k)}$ and $D^{\left( k\right) }$ in (\ref{Dabvk}) commute. This result follows from the fact
that the corresponding eigenvalues $\Lambda _{n}\left( E^{(k)}\right) $ and $\Lambda^{(k)}
_{n+k}$ in (\ref{autov_k}) commute, and the linear map that assigns to each differential operator in the algebra of differential operators  $D(W^{(k)})$ its corresponding sequence of eigenvalues, is an isomorphism (see \cite[Propositions 2.6 and 2.8]{GT07}).
% 
%ha\-ving  $\left( P_{n+k}^{\left( \alpha
%	,\beta ,v,k\right) } \right) _{n\geq 0}$ as common eigenfunctions

\end{remark}

\begin{remark}
The Darboux transform $\widetilde{E}^{(k)}= \eta^{(k)} \circ \dfrac{d}{dt}$ of the operator $E^{(k)}$ is not symmetric with respect to $W^{(k)}.$ Moreover, it is symmetric with respect to $W^{(k+1)}.$ Indeed,
\begin{eqnarray*}
\eta^{(k)} \circ \dfrac{d}{dt} &=&\frac{d^{2}}{dt^{2}} \left(\Phi ^{(k)}\left( t\right)\right)^{\ast}%
+\frac{d%
}{dt}\left(\frac{d%
}{dt} \left(\Phi ^{(k)}\left( t\right)\right)^{\ast} +\frac{d}{dt}\left(\Psi ^{(k)}\left( t\right)\right)^{\ast} \right) +\frac{d}{dt} 
\left(\Psi ^{(k)}\left( t\right)\right)^{\ast}\\
&=&\frac{d^{2}}{dt^{2}} \left(\mathscr{A}^{k}_{2}t^{2}+\mathscr{A}^{k}_{1}t+\mathscr{A}^{k}_{0}\right)^{\ast}%
+\frac{d}{dt}\left(\left( 2\mathscr{A}^{k}_{2}+\mathscr{B}^{k}_{1}\right) t+\mathscr{A}^{k}_{1}+\mathscr{B}^{k}_{0}\right)^{\ast}+(\mathscr{B}^{k}_{1})^{\ast}.
\end{eqnarray*}%
 In fact, if we substitute the coefficient of the second derivative in the first symmetry condition in (\ref{EDS1}) we obtain%
\begin{equation*}
W^{(k)}(t)\left( \mathscr{A}^{k}_{2}t^{2}+\mathscr{A}^{k}_{1}t+\mathscr{A}^{k}_{0}\right) =\left(
\mathscr{A}^{k}_{2}t^{2}+\mathscr{A}^{k}_{1}t+\mathscr{A}^{k}_{0}\right) ^{\ast }W^{(k)}\left( t\right),
\end{equation*}%
which does not hold.  Taking the main coefficient $W^{(k)}_2$ of $\widetilde W^{(\alpha\beta,v,k)}$ in (\ref{wk}), one has in particular 
\begin{equation*}
W^{k}_{2}\mathscr{A}^{k}_{1}-\left( \mathscr{A}^{k}_{1}\right) ^{\ast }W^{k}_{2}=\frac{4v(\alpha +\beta +2(k+1))}{%
(\kappa_{-v,\beta}+2(k+1))(\kappa_{v,\beta}+2(k+1))}%
\begin{pmatrix}
0 & 1 \\ 
-1 & 0%
\end{pmatrix}%
\neq \mathbf{0.}
\end{equation*}
The second statement follows from Proposition \ref{Shift}.
\end{remark}

\section{The algebra $D\left( W\right) $}\label{algebra}

In this section we will discuss some properties of the structure of the algebra
of matrix differential operators having as eigenfunctions a sequence of
polynomials $\left( P_{n}\right) _{n\geq 0}$, orthogonal with respect to the weight matrix $W=W^{\left( \alpha
	,\beta ,v\right) }$, i.e.,
\begin{eqnarray*}
	D\left( W\right) =\left\{ D:P_{n}D=\Lambda _{n}\left( D\right)P_{n} ,\text{ }%
	\Lambda _{n}\left( D\right) \in {%
		%TCIMACRO{\U{2102} }%
		%BeginExpansion
		\mathbb{C}
		%EndExpansion
	}^{N\times N}\text{ for all }n\geq 0\right\}.
\end{eqnarray*}
%where  is a sequence of matrix-valued
%polynomials orthogonal with respect to $W^{\left( \alpha .\beta ,v\right) }.$ 
%This subject has received much attention in the last few years (see )
The definition of $D(W)$ does not depend on the particular sequence of
orthogonal polynomials (see \cite[Corollary 2.5]{GT07}).

% To simplify the notation in this section we define the parameters $\kappa_{v,-\beta}= \alpha -\beta +v$ and $\kappa_{-v,-\beta}= \alpha -\beta -v$.
%{\color{red} extender esto al resto del paper...? si se incluyen aqui convendria usarlo en todo el paper}
\begin{theorem} Consider the weight matrix function $W=W^{(\alpha,\beta,v)}$(t). Then, the differential operators of order at most two in $D(W)$ are of the form 
	\begin{equation}
	D=\frac{d^{2}}{dt^{2}}\left( \mathcal{A}_{2}t^{2}+\mathcal{A}_{1}t+\mathcal{A}_{0}\right) +\frac{d}{dt}\left(
	\mathcal{B}_{1}t+\mathcal{B}_{0}\right) +\mathcal{C}_{0},  \label{Dg}
	\end{equation}%
	where%
	\begin{eqnarray*}
		\mathcal{A}_{2} &=&%
		\begin{pmatrix}
			a & c \\ 
			b & d%
		\end{pmatrix}%
		\text{ },\quad a,b,c,d\in 
		\mathbb{C}, \\
		\mathcal{A}_{1} &=&\frac{1}{2v}\left[
		\begin{pmatrix}
			-2va&(a-d)\kappa_{-v,-\beta}\\ (a-d)\kappa_{v,-\beta}&-2vd
		\end{pmatrix} +b\kappa_{-v,-\beta}\begin{pmatrix}-1&0\\2&1   \end{pmatrix}+c\kappa_{v,-\beta}\begin{pmatrix}-1&-2\\0&1   \end{pmatrix}\right],
		\\
		\mathcal{A}_{0} &=&\dfrac{ (a-d)\kappa_{v,-\beta}\kappa_{-v,-\beta}
			+b \kappa_{-v,-\beta} ^{2}-c\kappa_{v,-\beta} 
			^{2} }{%
			4v^{2}}%
		\begin{pmatrix}
			-1 & -1 \\ 
			1 & 1%
		\end{pmatrix}%
		,\\
		%\text{ \ with }%
		%&&
		%\begin{array}{c}
		%\theta =\dfrac{ \kappa_{v,-\beta}\kappa_{-v,-\beta}
		%a+ \kappa_{-v,-\beta} ^{2}b-\kappa_{v,-\beta} 
		%^{2}c-\kappa_{-v,-\beta} \kappa_{v,-\beta} d}{%
		%4v^{2}}%
		%\end{array}
		\\
		\mathcal{B}_{1} &=&%
		\begin{pmatrix}
			a\left( \alpha +\beta +4\right) & \left( \kappa_{-v,\beta} +4\right) c \\ \left( \kappa_{v,\beta} +4\right) b 
			& \left( \alpha +\beta +4\right) d%
		\end{pmatrix},
		\\
		\mathcal{B}_{0} &=&\frac{1}{4v}\left[
		a\begin{pmatrix}
			-4((\alpha+1) v-\kappa_{-v,-\beta}) &  \kappa_{-v,-\beta}
			\left( \kappa_{-v,\beta} +6\right)  \\ 
			\kappa_{v,-\beta} \left( \kappa_{v,\beta} +2\right) 
			& 0%
		\end{pmatrix}%
		+%
		b\kappa_{-v,-\beta}\begin{pmatrix}
			-\left( \kappa_{v,\beta} +2\right)
			&0
			\\ 
			2 \left(\kappa_{v,\beta} +4\right) & \kappa_{v,\beta} +6
		\end{pmatrix} \right.
		\\
		&&+\left.
		c\kappa_{v,-\beta}\begin{pmatrix}
			- \left( \kappa_{-v,\beta} +6\right) 
			& -2\left( \kappa_{-v,\beta} +4\right)  \\ 
			0
			&  \kappa_{-v,\beta} +2 %
		\end{pmatrix}%
		+%
		d\begin{pmatrix}
			0 &-\kappa_{-v,-\beta}\left( \kappa_{-v,\beta} +2\right) 
			\\ -\kappa_{v,-\beta} \left( \kappa_{v,\beta} +6\right) 
			
			& -4((\alpha+1)v+\kappa_{v,-\beta})%
		\end{pmatrix}\right],
		\\
		\mathcal{C}_{0} &=&\dfrac{1}{4}(\kappa_{v,\beta}+4 )(\kappa_{v,\beta}+2)
		\begin{pmatrix}
			a\dfrac{\left( \kappa_{-v,\beta} +4\right)}{\kappa_{v,\beta}+4} 
			-d\dfrac{ \kappa_{-v,\beta}+2}{\kappa_{v,\beta} +2} &c\dfrac{\left( \kappa_{-v,\beta} +4\right) (\kappa_{-v,\beta} +2)}{(\kappa_{v,\beta}+4 )
				(\kappa_{v,\beta} +2)} \\  b
			& 0%
		\end{pmatrix}%
		+eI,\quad e\in 
		\mathbb{C}.
	\end{eqnarray*}%
	%with $\gamma_0= and $
	%EndExpansion
	
\end{theorem}

\begin{proof}
	Let $\left( P_{n}^{\left( \alpha ,\beta ,v\right) }\right) _{n\geq 0}$ be the
	monic sequence of orthogonal polynomials with respect to $W^{\left( \alpha
		,\beta ,v\right) }.$ The polynomial $P_{n}^{\left( \alpha ,\beta ,v\right) }$ is an eigenfunction of the operator $D$ (\ref{Dg})
	if 
	\begin{equation*}
	P_{n}^{\left( \alpha ,\beta ,v\right) }D=\Lambda _{n}P_{n}^{\left( \alpha ,\beta
		,v\right) },
	\end{equation*}%
	with $\Lambda _{n}=n\left( n-1\right) \mathcal{A}_{2}+n\mathcal{B}_{1}+\mathcal{C}_{0}.$ This
	equation holds if and only if 
	\begin{equation}\label{algb}
	\begin{array}{c}
	k(k-1)\mathcal{P}_{n}^{k}\mathcal{A}_{2}+(k+1)k\mathcal{P}_{n}^{k+1}\mathcal{A}_{1}+\left( k+2\right) \left(
	k+1\right)\mathcal{P}_{n}^{k+2} \mathcal{A}_{0}+k\mathcal{P}_{n}^{k}\mathcal{B}_{1} \\ 
	+\left( k+1\right) \mathcal{P}_{n}^{k+1}\mathcal{B}_{0}+\mathcal{P}_{n}^{k}\mathcal{C}_{0}-\left( n\left(
	n-1\right) \mathcal{A}_{2}+n\mathcal{B}_{1}+\mathcal{C}_{0}\right)\mathcal{P}_{n}^{k} =0,%
	\end{array}%
	\end{equation}%
	where $\mathcal{P}_{n}^{k}$ denotes de $k-th$ coefficient of $P_{n}, \ k=0,1,2\ldots n.$
	
	To  prove the theorem we need  to solve   equation (\ref{algb}) for $k=n-1$ and $k=n-2$ to
	find relations between the parameters of the matrix-valued coefficients $%
	\mathcal{A}_{2},\mathcal{A}_{1},\mathcal{A}_{0},\mathcal{B}_{1},\mathcal{B}_{0}$ and $\mathcal{C}_{0}$. 
	
	%The resolution of these equation
	%are simple but require long computation. Therefore we will give the main
	%lines to reproduce them.
	
	%From (\ref{h21po}) 
	We obtain the explicit expressions of  $\mathcal{P}_{n}^{n-1}$ and $\mathcal{P}_{n}^{n-2}$ by substituting $k=0$ in the equalities (\ref{Pn-k-1}) and (\ref{Pn-k-2}) respectively.

	%\begin{eqnarray*}
	%\mathcal{P}_{n}^{n-1} &=&\dfrac{n}{v}%
	%\begin{pmatrix}
	%-\dfrac{(\alpha +n)v-\kappa_{-v,-\beta}}{\alpha +\beta +2(n+1)} &  \dfrac{\kappa_{-v,-\beta} }{\kappa_{v,\beta}+2(n+1)}\\ -\dfrac{%
	%\kappa_{v,-\beta}}{\kappa_{-v,\beta} +2(n+1)}
	% & -\dfrac{(\alpha
	%+n)v+\kappa_{v,-\beta}}{\alpha +\beta +2(n+1)}
	%\end{pmatrix}\\
	%\end{eqnarray*}
	%\mathcal{P}_{n}^{n-2} &=&
	%\begin{equation*}
	%\frac{n\left( n-1\right)(\alpha +n+1)}{v(\alpha +\beta+2n+2)}\mathscr{P}_{n}^{n-2},\ \textrm{with}\ \mathscr{P}_{n}^{n-2}=
	%%&=&
	%\end{eqnarray*}
	
	%with $$ $$
	%\begin{equation*}
	%\left( 
	%\begin{array}{cc}
	%\dfrac{\left[(\kappa_{v,\beta}+2n)(\alpha+n)+2(n+\beta+1)\right]-2(\alpha-\beta)(\alpha+\beta+2n+1))}{%
	%2(\alpha+\beta+2n+1)(\alpha+\beta+2n+v+2)} &-\frac{\kappa_{-v,-\beta}}{\alpha+\beta+2n+v+2} \\ \frac{\kappa_{v,-\beta}}{\alpha+\beta+2n-v+2}
	%&\frac{\left[(\kappa_{v,\beta}+2n)(\alpha+n)+2(n+\beta+1)\right]v+2(\alpha-\beta)(\alpha+\beta+2n+1))}{%
	%	2(\alpha+\beta+2n+1)(\alpha+\beta+2n-v+2)}
	%\end{array}%
	%\right)
	%\end{equation*} 
	%\end{eqnarray*}
	
	From equation (\ref{algb}) for $k=n-1$ we get%
	\begin{equation}
	\left( P_{n}^{n-1}\Lambda _{n}-\Lambda _{n}P_{n}^{n-1}\right)
	-P_{n}^{n-1}\left( 2\left( n-1\right) A_{2}+B_{1}\right) +\left[ n\left(
	n-1\right) A_{1}+nB_{0}\right] =0 . \label{k=n-1}
	\end{equation}%
	%Carriying out some calculations, one can see that
	Multiplying  equation (\ref{k=n-1}) by
	\begin{equation*}
	\frac{v\left( \alpha +\beta +2\left( n+1\right) \right) \left( \kappa
		_{v,\beta }+2\left( n+1\right) \right) \left( \kappa _{-v,\beta }+2\left(
		n+1\right) \right) }{n}
	\end{equation*}%
	one obtains a matrix polynomial on $n$ of degree four, where  each coefficient must be equal to zero.
	From the expression of the coefficient of $n^{4}$ we get the expression for $\mathcal{A}_{1}$ given above and from the expression of the coefficient of $n^{3}$ we get $\mathcal{B}_{0}$ in terms of  $\mathcal{A}_{2}$ and $\mathcal{B}_{1}$.
	Looking at  the entries $\left( 1,1\right) ,\left( 1,2\right) $ and $\left(
	2,2\right) $ of the coefficient of $n^{2}$ and the fact that $\kappa
	_{v,-\beta }$ and $\kappa _{-v,-\beta }$ are no zero we get $\left( \mathcal{C}_{0}\right)
	_{12},$ $\left( \mathcal{C}_{0}\right) _{11}$ and $\left( \mathcal{B}_{1}\right) _{12}$  respectively in terms of $\mathcal{A}_{2}$ and the other entries of $\mathcal{C}_{0}$ and $\mathcal{B}_{1}$.
	Finally, looking at the coefficient of $n$ we get the values of  $\left( \mathcal{B}_{1}\right) _{11},$ $%
	\left( \mathcal{B}_{1}\right) _{21},$ $\left( \mathcal{B}_{1}\right) _{22}$ and $\left( \mathcal{C}_0\right)
	_{21}$,  consequently,  we obtain the values of  $\mathcal{B}_{1}$, $\mathcal{B}_{0}$ and $\mathcal{C}_{0}$ written above.
	
	\medskip
	
	Analogously, from  equation (\ref{algb}) for $k=n-2$ we obtain%
	\begin{equation}
	\left( P_{n}^{n-2}\Lambda _{n-2}-\Lambda _{n}P_{n}^{n-2}\right) +\left(
	n-1\right) P_{n}^{n-1}\left( \left( n-2\right) A_{1}+B_{0}\right) +n\left(
	n-1\right) A_{0}=0.  \label{k=n-2}
	\end{equation}%
	%Doing some calculation is not difficult to show that the the product
	Multiplying equation (\ref{k=n-2}) by
	\begin{eqnarray*}
		&&v^{2}(\alpha +\beta +2n+1)(\alpha +\beta +2\left(
		n+1\right) )(\kappa _{v,\beta }+2\left( n+1\right) )\\
		&&(\kappa _{v,\beta }+\alpha +\beta +2\left( 2n+1\right) )(\kappa _{-v,\beta
		}+\alpha +\beta +2\left( 2n+1\right) )(\kappa _{-v,\beta }+2\left(
		n+1\right) )
	\end{eqnarray*}%
	one obtains  a matrix polynomial on $n$ of degree eight, where  each coefficient must be equal to zero. We get the expression of $%
	\mathcal{A}_{0}$ from the coefficient of $n^{8}$.
	
	Thus, if we replace the expressions of $\mathcal{A}_{0},\ \mathcal{A}_{1},\ \mathcal{B}_{1},\ \mathcal{B}_{0}$ and the
	entries $\left( 1,1\right) ,$ $\left( 1,2\right) $ and $\left( 2,1\right) $
	of $\mathcal{C}_0$ in (\ref{k=n-1}) and (\ref{k=n-2}) both equations hold true.

	%Now, using the expressions of $\mathcal{P}_{n}^{n-1}$ and $\mathcal{P}_{n}^{n-2}$ , we can write both equations as matrix
	%polynomials in $n$ of degree five and equal each coefficient to zero. From
	%equation (\ref{algb}) with $k=n-1$ we get the expressions for the matrices $\mathcal{A}_{1},\mathcal{B}_{1}$
	%and $\mathcal{B}_{0}$ and for the entries $\left( 1,2\right) ,$ $\left( 1,1\right) $
	%and $\left( 2,1\right) $ of $\mathcal{C}_{0},$ and from the other equation we get $%
	%\mathcal{A}_{0}.$
	
	Let $\mathcal{D}_{2}$ be the complex vector space of differential operators in $%
	D(W)$ of order at most two. We have already proved that $\dim \mathcal{D}_{2}\leq 5$.
	%{\color{red} 
	%Aqui seria mas bien decir que la dimension es 4 moudulus operators of lower order, igual que en \cite{CG06}. La identidad no cuenta, aunque entiendo se est\'a siguiendo exactamente el procedimiento de \cite{Z16}}
	
	If $D$ is symmetric then $D\in D\left( W\right) $. Using symmetry equations in  (\ref{EDS1}) one verifies that the operator $D$ in (\ref{Dg}) is  symmetric with respect to $W$ if
	and only if $a,d,e\in 
	%TCIMACRO{\U{211d} }%
	%BeginExpansion
	\mathbb{R}
	%EndExpansion
	$ and  condition
	
	%Therefore, we will
	%prove that differential operator $D$ of the form (\ref{Dg}) is  symmetric
	%with respect to $W$ if and only if  $a,d,e\in 
	%\mathbb{R}$ 
	% and 
	\begin{equation}\label{condconj}
	b\dfrac{\kappa_{v,\beta}+2}{\kappa_{v,-\beta}}=-\overline{c}\dfrac{\kappa_{-v,\beta}+2}{\kappa_{-v,-\beta}}
	\end{equation}
	holds true. Indeed, writing $W(t)=W^{\left( \alpha
		,\beta ,v\right) }=W_2t^2+W_1t+W_0$, from the first equation of symmetry in (\ref{EDS1}), we have that $W_{2}\mathcal{A}_{2}^{\ast
	}-\mathcal{A}_{2}W_{2}=0,$, i.e.,
	
	%{\color{red} Aqu\'i confundes a y b con $\alpha$ y $\beta$ Revisar!!!!!}
	\begin{equation}\label{1eqsim}
	\begin{pmatrix}
	2\textit{Im}\left( a\right) \dfrac{(\kappa_{v,\beta}+2)}{\kappa_{v,-\beta}} &-\overline{b}\dfrac{(\kappa_{v,\beta}+2)}{\kappa_{v,-\beta}}-c\dfrac{(\kappa_{-v,\beta}+2)}{\kappa_{-v,-\beta}} \\ 
	b\dfrac{(\kappa_{v,\beta}+2)}{%
		\kappa_{v,-\beta}}+\overline{c}\dfrac{(\kappa_{-v,\beta}+2)}{\kappa_{-v,-\beta}} & -2
	\textit{Im}\left( d\right) \dfrac{(\kappa_{v,\beta}+2)}{\kappa_{-v,-\beta}}%
	\end{pmatrix}%
	=0, 
	\end{equation}%
	where $Im(z)$ denotes the imaginary part of a complex number $z$.
	Then, since $\kappa_{v,\beta}+2>0$ because of the restrictions of the parameters $\alpha$, $\beta$ and $v$ in the definition of $W^{(\alpha,\beta,v)}$ in (\ref{W}), to verify (\ref{1eqsim}) one needs to have $a,d\in 
	%TCIMACRO{\U{211d} }%
	%BeginExpansion
	\mathbb{R}
	%EndExpansion
	$ and condition (\ref{condconj}).
	%{\color{red} >Qu\'e significa esa I multiplicando a Im(a)?}
	
	In addition, from the third symmetry equation (\ref{EDS1}) we have that $e\in 
	\mathbb{R}$.

	%To complete the proof of the theorem verify that the differential

	% 
	%Using (\ref{EDS1}) it is no difficult to prove.
	
	Thus, there exists at least five linearly independent symmetric operators of order at most two in $D(W)$. Therefore $\dim \mathcal{D}_{2}=5$.
	%{\color{red} Aqui seria mas bien decir que la dimension es 4 moudulus operator of lower order, igual que en \cite{CG06}. La identidad no cuenta, aunque entiendo estas siguiendo exactamente el procedimiento de \cite{Z16}}
\end{proof}

By taking as the only non zero parameters $a=1$ and $d=1$ respectively in the expression of the operator in (\ref{Dg}), we write the operators:\begin{eqnarray*}
	D_{1} &=&%
	\frac{d^{2}}{dt^{2}}\left[\begin{pmatrix}
		t^{2}-t&\dfrac{\kappa_{-v,-\beta}}{2v}t\\\dfrac{\kappa_{v,-\beta}}{2v}t&0\end{pmatrix}
	+\dfrac{\kappa_{v,-\beta}  \kappa_{-v,-\beta}}{4v^{2}} \begin{pmatrix}
		-1&-1\\1&1
	\end{pmatrix}\right]%
	\\
	&+&
	\frac{d}{dt}\begin{pmatrix}
		\left( \alpha +\beta +4\right) t+\dfrac{\kappa_{-v,-\beta}}{v}-(\alpha+1) & 
		\dfrac{\kappa_{-v,-\beta} \left( \kappa_{-v,\beta} +6\right) }{4v}\\
		\dfrac{\kappa_{v,-\beta} \left( \kappa_{v,\beta}+2\right) }{4v}& 0%
	\end{pmatrix}+
	\begin{pmatrix}
		\dfrac{\left(\kappa_{-v,\beta} +4\right) \left( \kappa_{v,\beta} +2\right) }{4}
		& 0 \\ 
		0 & 0%
	\end{pmatrix},\\
	D_{2} &=&%
	\frac{d^{2}}{dt^{2}}\left[ \begin{pmatrix}0&-\dfrac{\kappa_{-v,-\beta}}{2v}t\\
		-\dfrac{\kappa_{v,-\beta}}{2v}t&t^2-t\end{pmatrix}
	-\dfrac{\kappa_{-v,-\beta}\kappa_{v,-\beta}}{%
		4v^{2}} 
	\begin{pmatrix}
		-1&-1\\1&1
	\end{pmatrix}\right]%
	\\
	&+&
	\frac{d}{dt} \begin{pmatrix}
		0 &  
		-\dfrac{\kappa_{-v,-\beta} ( \kappa_{-v,\beta}+2) }{4v}
		\\-\dfrac{\kappa_{v,-\beta} \left( \kappa_{v,\beta}+6\right) }{%
			4v}& \left( \alpha +\beta +4\right) t-\dfrac{\kappa_{v,-\beta}}{v}-(\alpha +1)%
	\end{pmatrix}%
	+
	\begin{pmatrix}
		-\dfrac{1}{4}\left( \kappa_{v,\beta} +4\right) \left( \kappa_{-v,\beta} +2\right) 
		& 0 \\ 
		0 & 0%
	\end{pmatrix}.%
\end{eqnarray*}Analogously, by choosing as nonzero parameters $c =1$, $b=-\dfrac{\kappa_{v,-\beta}(\kappa_{-v,\beta}+2)}{\kappa_{-v,-\beta}(\kappa_{v,\beta}+2)}$ and  $c =i$, \\
$b=i\dfrac{\kappa_{v,-\beta}(\kappa_{-v,\beta}+2)}{\kappa_{-v,-\beta}(\kappa_{v,\beta}+2)}$ respectively, we define the operators:

%\vspace{-0.8cm}

\medskip

\begin{align*}
D_{3}=\frac{d^{2}}{dt^{2}}
	\left\{\left( 
	\begin{array}{cc}
		0 &1
		\\ 
		-\dfrac{\kappa_{v,-\beta}(\kappa_{-v,\beta}+2)}{\kappa_{-v,-\beta }(\kappa_{v,\beta}+2)}&0
	\end{array}
	\right)t^2 +\dfrac{\kappa_{v,-\beta}}{\kappa_{v,\beta}+2 }\left[ \left( 
	\begin{array}{cc}
		-1 & -\dfrac{\kappa_{v,\beta}+2  }{v}
		\\ 
		-\dfrac{\kappa_{-v,\beta}+2 }{v}&1
	\end{array}
	\right)t\right.\right.\\
	+\dfrac{1}{2}\left.\left.  \left(\dfrac{(\alpha +\beta+2)(\alpha-\beta)}{v^2}+1 \right)
	\begin{pmatrix}1&1\\
		-1&-1 
	\end{pmatrix} 
	\right]	\right\}+ \dfrac{d}{dt}\left[\begin{pmatrix}
		0 & \kappa_{-v,\beta}+4 \\ -\dfrac{\kappa_{v,-\beta}(\kappa_{-v,\beta}+2)(\kappa_{v,\beta}+4)}
		{\kappa_{-v,-\beta}(\kappa_{v,\beta}+2)} & 0%
	\end{pmatrix} t \right. \\
	   +\dfrac{\kappa_{v,-\beta}}{v} \left. \begin{pmatrix}
		-1 & 	-\dfrac{(\kappa_{-v,\beta}+4)}{2} \\ -\dfrac{(\kappa_{v,\beta}+4)(\kappa_{-v,\beta}+2)}{2(\kappa_{v,\beta}+2)} & -\dfrac{\kappa_{-v,\beta}+2}{\kappa_{v,\beta}+2}%
	\end{pmatrix}  \right]+ \dfrac{1}{4}(\kappa_{-v,\beta}+2) \begin{pmatrix} 0 & \kappa_{-v,\beta}+4 \\
		-\dfrac{(\kappa_{v,\beta}+4)\kappa_{v,-\beta}}{\kappa_{-v,-\beta}} & 0
	\end{pmatrix},
\end{align*}
\begin{align*}
iD_{4}=\frac{d^{2}}{dt^{2}} 
	\left\{ \left( 
	\begin{array}{cc}
		0 & -1
		\\ 
		-\dfrac{\kappa_{v,-\beta}(\kappa_{-v,\beta}+2)}{\kappa_{-v,-\beta}(\kappa_{v,\beta}+2)}&0
	\end{array}
	\right)t^2+\dfrac{\kappa_{v,-\beta}}{v(\kappa_{v,\beta}+2)}\left[\left( 
	\begin{array}{cc}
		\alpha+\beta+2 & \kappa_{v,\beta}+2
		\\ 
		-(\kappa_{-v,\beta}+2)&-(\alpha+\beta+2) \end{array}\right)t \right.\right.\\
	%	&& \qquad\qquad\qquad\qquad\qquad\qquad\qquad-
	\left. \left.- (\alpha +1)
	\begin{pmatrix}1&1\\
		-1&-1 
	\end{pmatrix} 
	\right]	\right\}
	+\dfrac{d}{dt} \left[\begin{pmatrix}
		0 & 
		-(\kappa_{-v,\beta}+4) \\ -\dfrac{\kappa_{v,-\beta}(\kappa_{-v,\beta}+2)(\kappa_{v,\beta}+4)}{\kappa_{-v,-\beta}(\kappa_{v,\beta}+2)} &0
	\end{pmatrix}t \right.\\
	+ \dfrac{\kappa_{v,-\beta}}{2v}  \left. \begin{pmatrix}
		\kappa_{-v,\beta}+4 & 
		\kappa_{-v,\beta}+4 \\ -\dfrac{(\kappa_{v,\beta}+4)(\kappa_{-v,\beta}+2)}{(\kappa_{v,\beta}+2)}&-\dfrac{(\kappa_{v,\beta}+4)(\kappa_{-v,\beta}+2)}{(\kappa_{v,\beta}+2)}%
	\end{pmatrix}\right]\\
-\dfrac{\kappa_{-v,\beta}+2}{4}\begin{pmatrix}
		0 & \kappa_{-v,\beta}+4\\
		\dfrac{(\kappa_{v,\beta}+4)\kappa_{v,-\beta}}{\kappa_{-v,-\beta}} & 0
	\end{pmatrix}.
\end{align*}
%\vspace{-1cm}

One has the following \begin{corollary}
	The set of symmetric operators $\left\{ D_{1},D_{2},D_{3},D_{4},I\right\} $ is a basis of the space of differential operators of order at most two in $D(W)$.
	Moreover, \ the corresponding eigenvalues for the differential operators $%
	D_{1},D_{2},D_{3}$ and $D_{4}$ are%
	
		\begin{eqnarray*}
		\Lambda _{n}\left( D_{1}\right)  &=&\frac{1}{4}%
		\begin{pmatrix}
			\left( \kappa_{v,\beta}+2(n+1) \right) \left(\kappa_{-v,\beta}+2(n+2) \right)  & 0
			\\ 
			0 & 0%
		\end{pmatrix},
		\\
		\Lambda _{n}\left( D_{2}\right)  &=&%
		\begin{pmatrix}
			-\dfrac{1}{4}\left(\kappa_{-v,\beta}+2 \right) \left( \kappa_{v,\beta}+4
			\right)  & 0 \\ 
			0 & \left( n+\alpha +\beta +3\right) n%
		\end{pmatrix},
		\\
		\Lambda _{n}\left( D_{3}\right)  &=&%
		\dfrac{1}{4}\left( \kappa_{-v,\beta}+2(1+n) \right) \left(\kappa_{-v,\beta}+ 2(2+n)
		\right)
		\begin{pmatrix}
			0 & 
			1 \\
			0& 0%
		\end{pmatrix}\\
		&-&\dfrac{\left( \kappa_{v,\beta}+2(1+n)\right) \left(\kappa_{v,\beta}+2(2+n)\right) \left( \kappa_{-v,\beta}+2\right) \kappa_{v,-\beta} }{%
			4\kappa_{-v,-\beta} \left(\kappa_{v,\beta}+2\right) } \begin{pmatrix}
			0 & 0 \\
			1& 0%
		\end{pmatrix},
		\\
		\Lambda _{n}\left( iD_{4}\right)  &=&%
		-\dfrac{1}{4}\left( \kappa_{-v,\beta}+2(1+n) \right) \left(\kappa_{-v,\beta}+ 2(2+n)
		\right)
		\begin{pmatrix}
			0 & 
			1 \\
			0& 0%
		\end{pmatrix}\\
		&-&\dfrac{\left( \kappa_{v,\beta}+2(1+n)\right) \left(\kappa_{v,\beta}+2(2+n)\right) \left( \kappa_{-v,\beta}+2\right) \kappa_{v,-\beta} }{%
			4\kappa_{-v,-\beta} \left(\kappa_{v,\beta}+2\right) } \begin{pmatrix}
			0 & 0 \\
			1& 0%
		\end{pmatrix}.
	\end{eqnarray*}
\end{corollary}

\begin{corollary}
	The differential operators appearing in (\ref{Dalpha_beta}) and (\ref{OE}) are $%
	D^{\left( \alpha ,\beta ,v\right) }=-D_{1}-D_{2}$ and $E^{(0)}=-\dfrac{\kappa_{v,\beta}+4}{\kappa_{v,\beta}+2}D_{1}-\dfrac{\kappa_{-v,\beta}+4}{\kappa_{-v,\beta}+2}D_{2}$ respectively.
\end{corollary}

\begin{corollary}
	There are no operators of order one in the algebra D(W).
\end{corollary}

%{\color{blue} modificar a partir de la elecciÃ³n de la base}
\begin{proof}
	Suppose that there exists a right differential operator of order one, such that $D= aD_{1}+bD_{2}+cD_{3}+d(iD_{4})+eI$, with  $a,b,c,d,e \in \mathbb{R}$. Equating to zero the matrix-valued coefficient of $\dfrac{d^{2}}{dt^{2}}$ one obtains:

	\begin{equation*}
	\begin{pmatrix}
	a & c-di \\ 
	-\dfrac{\kappa_{v,-\beta}(\kappa_{-v,\beta}+2 )}{\kappa_{-v,-\beta}(\kappa_{v,\beta}+2)}(c+di) & b%
	\end{pmatrix}
	=\mathbf{0}.
	\end{equation*}%
	Therefore $a=b=c=d=0.$
\end{proof}

\begin{corollary}
	The algebra $D\left( W\right) $ is not commutative.
\end{corollary}

\begin{proof}
	Using the isomorphism  between the algebra of differential operators and the algebra of matrix-valued functions of $n$ generated by the eigenvalues going with this operators we have that
	$D_{1}D_{3}\neq D_{3}D_{1}$ \ since $\Lambda _{n}\left( D_{1}\right)
	\Lambda _{n}\left( D_{3}\right) \neq \Lambda _{n}\left( D_{3}\right) \Lambda
	_{n}\left( D_{1}\right) $.
\end{proof}

\begin{remark}
	In \cite{PZ16} the authors study the algebra $D\left( W^{\left( p,q\right) }\right) $, where $W^{\left( p,q\right) }$ is, for $p\neq \dfrac{q}{2}$,
	the irreducible  weight matrix%
	\begin{equation*}
	W^{\left( p,q\right) }(t)=\left( t\left( 1-t\right) \right) ^{\dfrac{q-2}{2}}%
	\begin{pmatrix}
	2pt^{2}-2pt+\dfrac{q}{2} & qt-\dfrac{q}{2} \\ 
	qt-\dfrac{q}{2} & -2\left( p-q\right) t^{2}+2\left( p-q\right) t+\dfrac{q}{2}
	\end{pmatrix},\quad t \in [0,1].
	\end{equation*}%
	Let us denote by $D_{1}^{\left( p,q\right) },D_{2}^{\left( p,q\right)
	},D_{3}^{\left( p,q\right) }$ and $D_{4}^{\left( p,q\right) }$ the
	differential operators appearing in \cite{PZ16}. Then, taking $\alpha=\beta=\dfrac{q}{2}-1$ in (\ref{W}) and writing $v=2p-q$,
	%$$q=2(\alpha+1),\quad p=\dfrac{v+q}{2}, $$
	we have the following
	relations with the operators $D_i,\ i=1,2,3,4$, defined above:
	\begin{eqnarray*}
		D_{1}^{\left( p,q\right) } &=&D_{1},\text{ \ \ \ }D_{2}^{\left(
			p,q\right) }=D_{2}+\left( q-p\right) \left( p+1\right) I ,\\
		D_{3}^{\left( p,q\right) } &=&\frac{p}{2(q-p) }\left(
		D_{3}+iD_{4}\right) ,\text{ \ \ \ \ }D_{4}^{\left( p,q\right) }=\frac{%
			1}{2}\left( D_{3}-iD_{4}\right) .
	\end{eqnarray*}
\end{remark}

%\renewcommand\bibname{References} %UNCOMMENT LINE TO INCLUDE BIBLIOGRAPHY
%\bibliographystyle{abbrv}
%\bibliography{bibliografia}

\begin{thebibliography}{KdlRR17}
	\bibitem{Atk}  F. V. Atkinson,
	\newblock Discrete and continous boundary problems,
\newblock 	Academic Press, New York, 1964.
	
	\bibitem{Be}  Ju. M.Berezanskii, 
	\newblock Expansions in
		eigenfunctions of selfadjoint
		operators, 
		\newblock {\em Transl. Math. Monographs}, AMS  {\bf 17} 1968.
\bibitem{BCD} J. Borrego, M. Castro,  and A. J. Dur\'an,
\newblock Orthogonal matrix polynomials satisfying differential equations with recurrence coefficients having non-scalar limits,
\newblock {\em Integral Transforms Spec. Funct.}, 23, 685- 700 (2012).

\bibitem{CGPSZ19}
C. Calder{\'o}n, Y. Gonz{\'a}lez, I. Pacharoni, S. Simondi, and I. Zurri{\'a}n,
\newblock $2\times2$ hypergeometric operators with diagonal eigenvalues
\newblock {\em J. Approx. Theory}, 248:105299, 17pp (2019).

\bibitem{CMV07}
M. J. Cantero, L. Moral and L. Vel{\'a}zquez,
\newblock Matrix orthogonal polynomials whose derivatives are also orthogonal,
\newblock {\em J. Approx. Theory}, 146(2):174--211 (2007).

\bibitem{C18}
W.~R. Casper,
\newblock Elementary examples of solutions to {B}ochner's problem for matrix
  differential operators,
\newblock {\em J. Approx. Theory},  229, 36--71 (2018).

\bibitem{C19}
W.~R. Casper, The symmetric $2\times 2$ hypergeometric matrix differential operators,
\newblock \newblock {\em arXiv:1907.12703},  (2019).


\bibitem{CY18}
W.~R. Casper and M.~Yakimov,
\newblock The matrix {B}ochner problem,
\newblock {	
	{\em Amer. J. Math.}  (2021), to appear, \em arXiv:1803.04405}.
%{\color{red} actualizar}

%\bibitem[CG05]{CG05}
%M.~Castro and F.~A. Gr{\"u}nbaum.
%\newblock {Orthogonal matrix polynomials satisfying first order differential
%  equations: a collection of instructive examples}.
%\newblock {\em J. Nonlinear Math. Physics}, 12(2):63--67, 2005.
\bibitem{CG06} M.M. Castro and F. A. Gr\"unbaum, 
\newblock The algebra of differential operators associated to a family of
	matrix-valued orthogonal polynomials: five instructive examples,
\newblock {\em Int. Math. Res. Not.}, 7,  1--33 (2006).
\bibitem{CG15}
M.~{Castro} and F.~A. {Gr\"unbaum},
\newblock {The Darboux process and time-and-band limiting for matrix orthogonal
  polynomials,}
\newblock {\em {Linear Algebra Appl.}},  487, 328--341  (2015).

\bibitem{CG17}
M.~Castro and F.~A. Gr\"unbaum,
\newblock Time-and-band limiting for matrix orthogonal polynomials of Jacobi
  type,
\newblock {\em Random Matrices: Theory and Applications}, 06(04):1740001, 12pp (2017).



\bibitem{CGPZ17}
M.~Castro, F.~A. Gr\"unbaum, I.~Pacharoni, and I.~Zurri{\'a}n,
\newblock A further look at time-and-band limiting for matrix orthogonal
  polynomials,
\newblock in M.~Nashed and X.~Li, editors, {\em Frontiers in Orthogonal
  Polynomials and $q$-Series}. World Scientific, 2018.
  
  


\bibitem{DG86}
J.~J. Duistermaat and F.~A. Gr{\"u}nbaum,
\newblock Differential equations in the spectral parameter,
\newblock {\em Comm. Math. Phys.}, 103(2):177--240 (1986).

\bibitem{Dur96}
A. J. Dur\'an,
\newblock Markov's theorem for orthogonal matrix polynomials,
\newblock {\em Canad. J. Math.}, 48:1180–1195 (1996). 


\bibitem{D97}
A.~J. Dur\'an,
\newblock {Matrix inner product having a matrix symmetric second-order
	differential operator},
\newblock {\em Rocky Mt. J. Math.}, 27(2):585--600 (1997).
  \bibitem{D2} A. J. Dur\'an, 
   \newblock Generating orthogonal matrix polynomials
  	satisfying se\-cond order differential equations from a trio of triangular matrices,
  \newblock{\em J. Approx. Theory}, 161,  88-113 (2009).
  
  \bibitem{D3} A. J. Dur\'an,  
    \newblock A method to find weight matrices having symme\-tric second-order differential
  	operators with matrix leading coefficents, {\em Const. Approx}, 29, 181-205 (2009).
  
\bibitem{D10}
A. J. Dur\'{a}n,
 \newblock Rodrigues' formulas for orthogonal matrix polynomials satisfying second-order differential equations,
 \newblock {\em Int. Math. Res. Not}, 5, 824--855 (2010).
 
 \bibitem{DG04}
 A.~J. Dur\'an and F.~A. Gr{\"u}nbaum,
 \newblock Orthogonal matrix polynomials satisfying second-order differential
 equations,
 \newblock {\em Int. Math. Res. Not.}, 10:461--484 (2004).

\bibitem{DG3}  A. J.  Dur\'an and F. A. Gr\"unbaum,
\newblock Structural formulas for
	orthogonal matrix polynomials satisfying
	second order differential equations, I,
\newblock {\em Constr. Approx.}, 22, 255--271 (2005).



\bibitem{DG5} A. J.  Dur\'an  and F. A. Gr\"unbaum, 
\newblock Matrix orthogonal polynomials satisfying
	second order differential equations: coping without help from group representation theory,
\newblock {\em J. Approx. Theory},  148, 35--48 (2007).

\bibitem{DdI1} A. J. Dur\'an,  and M. D. de la Iglesia, 
\newblock Some examples of orthogonal matrix polynomials satisfying
	odd order differential equations,
	\newblock{\em J. Approx. Theory, } 150, 153-174 (2008).

%\bibitem[DdI2]{DdI2}
%A. J. Dur\'an and de la Iglesia M. D.,
%{\it Second order differential operators having
%	several families of orthogonal matrix polynomials as eigenfunctions},
%Int. Math. Res. Not. (2008), Art ID rnn 084.
\bibitem{DL1} A. J. Dur\'an and P. L\'opez, 
\newblock Orthogonal matrix polynomials,
\newblock  {\em Laredo Lectures on Orthogonal Polynomials and Special Functions}, 13--44, Adv. Theory Spec. Funct. Orthogonal Polynomials, Nova Sci. Publ., Hauppauge, NY, 2004.
\bibitem{DL}  A. J. Dur\'an  and P. L{\'o}pez,
\newblock Structural formulas for orthogonal matrix polynomials satisfying
	second order differential equations, II, 
	\newblock {\em Constr. Approx.}, 26, No. 1,  29--47 (2007).





%\bibitem{DLR96}
%A.~J. Dur\'an and P.~Lopez-Rodriguez.
%\newblock Orthogonal matrix polynomials: zeros and {B}lumenthal's theorem.
%\newblock {\em J. Approx. Theory}, 84(1):96--118, 1996.


\bibitem{G03}
F.~A. Gr{\"u}nbaum,
\newblock Matrix-valued {J}acobi polynomials,
\newblock {\em Bull. Sci. Math.}, 127(3):207--214 (2003).

\bibitem{GdI}  F.A. Gr\"unbaum and  M. D. de la Iglesia,
\newblock Matrix-valued orthogonal polynomials
	arising from group representation theory and a family of quasi-birth-and death
	processes,
\newblock {\em SIAM J. Matrix Anal. Appl.}, 30, 741--761 (2008).

\bibitem{GPT01}
F.~A. Gr{\"u}nbaum, I.~Pacharoni, and J.~Tirao,
\newblock A matrix-valued solution to {B}ochner's problem,
\newblock {\em J. Phys. A}, 34(48):10647--10656 (2001).

\bibitem{GPT02a}
F.~A. Gr{\"u}nbaum, I.~Pacharoni, and J.~Tirao,
\newblock Matrix-valued spherical functions associated to the complex
  projective plane,
\newblock {\em J. Funct. Anal.}, 188(2):350--441 (2002).

\bibitem{GPT03}
F.~A. Gr{\"u}nbaum, I.~Pacharoni, and J.~Tirao,
\newblock Matrix-valued orthogonal polynomials of the {J}acobi type,
\newblock {\em Indag. Math. (N.S.)}, 14(3-4):353--366 (2003).

%\bibitem{GPT05}
%F.~A. Gr{\"u}nbaum, I.~Pacharoni, and J.~Tirao,
%\newblock matrix-valued orthogonal polynomials of {J}acobi type: the role of
 % group representation theory.
%\newblock {\em Ann. Inst. Fourier (Grenoble)}, 55(6):2051--2068, 2005.

\bibitem{GPZ15}
F.~A. {Gr\"unbaum}, I.~{Pacharoni}, and I.~{Zurri\'an},
\newblock Time and band limiting for matrix-valued functions, an example,
\newblock {\em {SIGMA, Symmetry Integrability Geom. Methods Appl.}},
  11(044):1--14 (2015).

\bibitem{GPZ17}
F.~A. Gr{\"u}nbaum, I.~Pacharoni, and I.~Zurri{\'a}n,
\newblock Time and band limiting for matrix-valued functions,
\newblock {\em Inverse Problems}, 33(2):1--26 (2017).

\bibitem{GPZ18}
F.~A. Gr{\"u}nbaum, I.~Pacharoni, and I.~Zurri{\'a}n,
\newblock Bispectrality and time-band-limiting: Matrix-valued polynomials,
\newblock {\em International Math. Research Notices} 13, 4016-4036 (2020).
%\newblock to appear.



\bibitem{GT07}
F.~A. Gr{\"u}nbaum and J.~Tirao,
\newblock The algebra of differential operators associated to a weight matrix,
\newblock {\em Integral Equations Operator Theory}, 58(4):449--475 (2007).

\bibitem{KRR17}
E.~Koelink, A.~de~los R{\'\i}os, and P.~Rom\'an,
\newblock Matrix-valued {G}egenbauer-type polynomials,
\newblock {\em Constr. Approx.}, 46(3):459--487 (2017).

\bibitem{KPR12}
E.~Koelink, M.~van Pruijssen, and P.~Rom\'an,
\newblock {Matrix-valued orthogonal po\-ly\-no\-mials related to
  $(\mathrm{SU}(2)\times \mathrm{SU}(2), diag)$},
\newblock {\em Int. Math. Res. Not.}, 2012(24):5673--5730 (2012).

\bibitem{KPR13}
E.~Koelink, M.~van Pruijssen, and P.~Rom\'an,
\newblock Matrix-valued orthogonal po\-ly\-no\-mials related to
  $(\mathrm{SU}(2)\times \mathrm{SU}(2),\mathrm{SU}(2))$, {II},
\newblock {\em PRIMS}, 49(2):271--312 (2013).

\bibitem{K49}
M.~G. Krein,
\newblock Infinite j-matrices and a matrix moment problem,
\newblock {\em Dokl. Akad. Nauk SSSR}, 69(2):125--128 (1949).

\bibitem{K71}
M.~G. Krein,
\newblock Fundamental aspects of the representation theory of hermitian
  operators with deficiency index $(m,m)$,
\newblock {\em AMS Translations, series 2}, 97:75--143 (1971).

\bibitem{LM}
L. Miranian,
\newblock On classical orthogonal polynomials and differential operators,
\newblock {\em J. Phys. A: Math Gen,} 38, 6379-6383 (2005).

\bibitem{PR08}
I.~Pacharoni and P.~Rom\'an,
\newblock {A sequence of matrix-valued orthogonal polynomials associated to
  spherical functions},
\newblock {\em Constr. Approx.}, 28(2):127--147 (2008).

\bibitem{PT06}
I.~Pacharoni and J.~Tirao,
\newblock Matrix-valued orthogonal polynomials arising from the complex
  projective space,
\newblock {\em Constr. Approx.}, 25(2):177--192 (2006).

%\bibitem{PTZ14}
%I.~{Pacharoni}, J.~{Tirao}, and I.~{Zurri\'an},
%\newblock {Spherical functions associated with the three-dimensional sphere}.
%\newblock {\em {Ann. Mat. Pura Appl. (4)}}, 193(6):1727--1778, 2014.

\bibitem{PZ16}
I.~Pacharoni and I.~Zurri{\'a}n,
\newblock Matrix {G}egenbauer {P}olynomials: {T}he {$2\times 2$} {F}undamental
  {C}ases,
\newblock {\em Constr. Approx.}, 43(2):253--271 (2016).

\bibitem{Sz} G. Szeg\"o, 
\newblock Orthogonal Polynomials,
\newblock{\em Coll. Publ., XXIII, American Mahematical Society}, Providence, RI, 1975.

\bibitem{T03}
J.~Tirao,
\newblock The matrix-valued hypergeometric equation,
\newblock {\em Proc. Natl. Acad. Sci. U.S.A.}, 100(14):8138--8141 (2003).

\bibitem{T11}
J.~Tirao,
\newblock {The algebra of differential operators associated to a weight matrix:
  a first example},
\newblock {Polcino Milies, C\'esar (ed.), Groups, algebras and applications.
  XVIII Latin American algebra colloquium, S\~ao Pedro, Brazil, August 3--8,
  2009. Proceedings. Providence, RI: American Mathematical Society (AMS).
  Contemporary Mathematics 537, 291-324} 2011.

\bibitem{TZ16}
J.~Tirao and I.~Zurri\'an,
\newblock {Reducibility of Matrix Weights},
\newblock {\em Ramanujan J.}, 45, no. 2, 349-374 (2018).
%\newblock DOI:10.1007/s11139-016-9834-9.

\bibitem{Z16}
I.~{Zurri{\'a}n},
\newblock {The Algebra of Differential Operators for a Gegenbauer Weight
  Matrix},
\newblock {\em Int. Math. Res. Not.}, 8, 2402-2430 (2016).

%\bibitem {KRR17}
%E. ~Koelink ,A. ~de los Rios and P. ~Rom{\'a}n,
 % \newblock {Matrix-valued Gegenbauer-type polynomials}.

  %\newblock {\em Constr. Approx.}, 46(3):459--487, 2017

\end{thebibliography}

\section{Aknowledgements}

The authors would like to thank Pablo Rom\'an for useful  suggestions on earlier versions of the paper. The authors also thank the anonymous referees for their careful reading and remarks.

\end{document}